\documentclass[11pt]{amsproc}

\usepackage{enumerate}

\usepackage{xargs}
\usepackage{amsmath,amssymb,amsfonts,amsthm}
\usepackage{mathtools}
\usepackage{tikz-cd}

\usepackage[backrefs]{amsrefs}

\newcommandx{\set}[2][2=\empty]{\{#1\ifx#2\empty\else\,|\,#2\fi\}}
\newcommand{\card}[1]{\lvert#1\rvert}
\newcommand{\nats}{\mathbb{N}}
\newcommand{\ints}{\mathbb{Z}}
\newcommand{\rats}{\mathbb{Q}}
\newcommand{\reals}{\mathbb{R}}
\newcommand{\complex}{\mathbb{C}}
\newcommand{\finfield}{\mathbb{F}}
\newcommand{\abs}[1]{\lvert#1\rvert}
\newcommand{\ceil}[1]{\lceil#1\rceil}
\DeclareMathOperator{\rchar}{char}
\DeclareMathOperator{\ord}{ord}
\DeclareMathOperator{\tr}{tr}
\DeclareMathOperator{\N}{N}
\DeclareMathOperator{\id}{id}
\DeclareMathOperator{\diag}{diag}
\DeclareMathOperator{\rk}{rk}
\DeclareMathOperator{\codim}{codim}
\newcommand{\rest}[1]{\left.#1\right\rvert}
\DeclareMathOperator{\im}{im}
\newcommandx{\scprod}[2][2=\empty]{\langle#1\ifx#2\empty\else\,|\,#2\fi\rangle}
\DeclareMathOperator{\sgn}{sgn}
\DeclareMathOperator{\supp}{supp}
\DeclareMathOperator{\End}{End}
\DeclareMathOperator{\rad}{rad}
\newcommandx{\grp}[2][2=\empty]{{\langle#1\ifx#2\empty\else\,|\,#2\fi\rangle}}
\newcommandx{\gensubgrp}[2][2=\empty]{\langle#1\ifx#2\empty\else\,|\,#2\fi\rangle}
\newcommandx{\gennorsubgrp}[2][2=\empty]{\langle\!\langle#1\ifx#2\empty\else\,|\,#2\fi\rangle\!\rangle}
\newcommandx{\gensubsp}[2][2=\empty]{\langle#1\ifx#2\empty\else\,|\,#2\fi\rangle}
\newcommandx{\gensubalg}[2][2=\empty]{\langle#1\ifx#2\empty\else\,|\,#2\fi\rangle}
\DeclareMathOperator{\Cycgrp}{C}
\DeclareMathOperator{\Sym}{Sym}
\DeclareMathOperator{\Sgrp}{S}
\DeclareMathOperator{\GL}{GL}
\DeclareMathOperator{\Sp}{Sp}
\DeclareMathOperator{\GO}{GO}
\DeclareMathOperator{\U}{U}
\DeclareMathOperator{\GU}{GU}
\DeclareMathOperator{\SL}{SL}
\DeclareMathOperator{\SU}{SU}
\DeclareMathOperator{\PSL}{PSL}
\DeclareMathOperator{\UT}{UT}
\newcommand{\freegrp}{\mathbf{F}}
\DeclareMathOperator{\Aut}{Aut}
\DeclareMathOperator{\Cay}{Cay}
\newcommand{\Proj}{\mathbb{P}}

\theoremstyle{plain}
\newtheorem{lemma}{Lemma}
\newtheorem{theorem}{Theorem}

\newtheorem{corollary}{Corollary}
\newtheorem{proposition}{Proposition}
\newtheorem{fact}{Fact}

\theoremstyle{definition}
\newtheorem{remark}{Remark}

\title[Word images are dense]{Word images in symmetric and classical groups\\ of Lie type are dense}
\author{Jakob Schneider}
\address{J.~Schneider, TU Dresden, 01062 Dresden, Germany}
\email{jakob.schneider@tu-dresden.de}
\author{Andreas Thom}
\address{A.~Thom, TU Dresden, 01062 Dresden, Germany}
\email{andreas.thom@tu-dresden.de}

\begin{document}
\begin{abstract}
	Let $w\in\freegrp_k$ be a non-trivial word and denote by $w(G)\subseteq G$ the image of the associated word map $w\colon G^k\to G$. Let $G$ be one of the finite groups $\Sgrp_n,\GL_n(q),\Sp_{2n'}(q),\GO_{2n'}^\pm(q),\GO_{2n'+1}(q),\GU_n(q)$ ($q$ a prime power, $n\geq 2$, $n'\geq 1$), or the unitary group $\U_n$ over $\complex$. Let $d_G$ be the normalized Hamming metric resp.\ the normalized rank metric on $G$ when $G$ is a symmetric group resp.\ one of the other classical groups and write $n(G)$ for the degree of $G$.

	For $\varepsilon>0$, we prove that there exists an integer $N(\varepsilon,w)$ such that $w(G)$ is $\varepsilon$-dense in $G$ with respect to the metric $d_G$ if $n(G)\geq N(\varepsilon,w)$. This confirms metric versions of conjectures by Shalev and Larsen.
	Equivalently, we prove that any non-trivial word map is surjective on a metric ultraproduct of groups $G$ from above such that $n(G)\to\infty$ along the ultrafilter.

	As a consequence of our methods, we also obtain an alternative proof of the result of Hui, Larsen, and Shalev that $w_1(\SU_n)w_2(\SU_n)=\SU_n$ for non-trivial words $w_1,w_2\in\freegrp_k$ and $n$ sufficiently large.
\end{abstract}

\maketitle

\tableofcontents

\section{Introduction}
 
Recently, there has been increasing interest in word maps on finite, algebraic, and topological groups \cite{avnigelanderkassabovshalev2013word,bandmangariongrunewald2012surjectivity,elkasapythom2014goto,garionshalev2009commutator,gordeevkunyavskiiplotkin2016word,gordeevkunyavskiiplotkin2018word,guralnickliebeckobrienshalevtiep2018surjective,guralnicktiep2015effective,huilarsenshalev2015waring,klyachkothom2017new,larsen2004word,larsenshalev2009word,larsenshalev2016distribution,larsenshalev2017words,larsenshalevtiep2012waring,lubotzky2014images}. Recall that, for a word $w\in\freegrp_k$, where $\freegrp_k$ denotes the free group of rank $k$ freely generated by $x_1,\ldots,x_k$,  and a group $G$, the symbol $w(g_1,\ldots,g_k)$ denotes the evaluation at $w$ of the homomorphism $\freegrp_k\to G$ which is defined by $x_i\mapsto g_i$ for $i=1,\ldots,k$. We call the map $G^k\to G$ which sends $(g_1,\ldots,g_k)\in G^k$ to $w(g_1,\ldots,g_k)\in G$ the word map associated to $w$ and write $w(G)\subseteq G$ for its image. 

Subsequently, fix a non-trivial word $w\in\freegrp_k$. For a fixed finite group $G$ the word image $w(G)$ can just be $\set{1_G}$ if $w$ is a law for $G$ -- when $G$ is a finite simple group, the possible word images were characterized by Lubotzky in \cite{lubotzky2014images}. In analogy, for a fixed compact group $G$, the word image $w(G)$ can be contained in any neighborhood of the identity as was proved by the second author in \cite[Corollary~1.2]{thom2013convergent}. However, examples show that for fixed $w$ and a family $\mathcal G$ of finite simple groups resp.\ compact connected simple Lie groups, for $G\in\mathcal G$, one should expect $w(G)$ to be \emph{large} in $G$ if the order resp.\ dimension or rank of $G$ is large.

There are two intriguing conjectures regarding this observation.
Letting $\mathcal G$ be the class of finite non-abelian simple groups, Shalev conjectured \cite[Conjecture 8.3]{bandmangariongrunewald2012surjectivity} that if $w$ is not a proper power, the associated word map on $G$ is surjective if the order of $G$ is sufficiently large. 
Similarly, if $\mathcal G$ is the class of simple connected compact groups, Larsen conjectured at the 2008 Meeting of the AMS in Bloomington that $w$ is surjective on $G\in\mathcal G$ if the rank of $G$ is sufficiently large.

Shalev's conjecture was disproved for groups of type $\PSL_2(q)$ in \cite{jamborliebeckobrien2013some} using trace polynomials, however it remains plausible that such word maps are surjective once the rank is large enough (as conjectured in \cite[Conjecture~4.6]{liebeck2015width}). Remarkably, Lyndon proved this for infinite symmetric groups, see~\cite{doughertymycielski1999representations}. A weak form of Larsen's conjecture (surjectivity on $\SU_n$ for infinitely many $n \in \mathbb N$) was proved by Elkasapy and the second author in \cite{elkasapythom2014goto} for all words $w\in\freegrp_2\setminus [[\freegrp_2,\freegrp_2],[\freegrp_2,\freegrp_2]]$. 

In this note, we prove that metric versions of Shalev's and Larsen's conjectures are true. Let us explain what we mean by this.
For $\varepsilon>0$ and a subset $Y\subseteq X$ of a metric space $(X,d)$, say $Y$ is $\varepsilon$-dense in $X$ if $d(x,Y)\coloneqq\inf_{y\in Y}{d(x,y)}\leq\varepsilon$ for all $x\in X$.
Throughout let $\Sgrp_n=\Sym(\underline{n})$ denote the symmetric group acting on the set $\underline{n}=\set{1,\dots,n}$. Denote by $d_{\rm H}\colon \Sgrp_n\times \Sgrp_n\to[0,1]$ the normalized Hamming metric, i.e.,
$$
d_{\rm H}(\sigma,\tau)\coloneqq\frac{1}{n}\card{\set{x\in\underline{n}}[x.\sigma\neq x.\tau]}
$$
for $\sigma,\tau\in \Sgrp_n$. 
The following metric analog of Shalev's conjecture holds for symmetric groups.

\begin{theorem}\label{thm:main_sym_grps}
Let $w\in\freegrp_k$ be a non-trivial word and $\varepsilon>0$. There exists an integer $N(\varepsilon,w)$ such that $w(\Sgrp_n)$ is $\varepsilon$-dense in $\Sgrp_n$ with respect to the normalized Hamming metric if $n \geq N(\varepsilon,w)$.
\end{theorem}

For a classical group $G$ of Lie type write $n=n(G)$ for the dimension of its natural module (i.e., its degree). Also define the normalized rank metric $d_{\rk}\colon G\times G\to[0,1]$ on $G$ by $d_{\rk}(g,h)\coloneqq\rk(g-h)/n$ for $g,h\in G$.
In analogy to Theorem~\ref{thm:main_sym_grps}, we then have the following.

\begin{theorem}\label{thm:main_thm_fin_grp_lt}
	Let $w\in\freegrp_k$ be a non-trivial word and $\varepsilon>0$. Let $G$ be one of the groups $\GL_n(q)$, $\Sp_{2n'}(q)$, $\GO_{2n'+1}(q)$, $\GO_{2n'}^{\pm}(q)$, or $\GU_n(q)$ ($q$ a prime power, $n\geq 2$, $n'\geq 1$). There exists an integer $N(\varepsilon,w)$ such that $w(G)$ is $\varepsilon$-dense in $G$ with respect to the normalized rank metric if $n=n(G)\geq N(\varepsilon,w)$.
\end{theorem}

Write $\U_n$ for the unitary group (over the complex numbers) of degree $n\in\ints_+$. Define the normalized rank metric on $\U_n$ as above. We will also prove the following metric version of Larsen's conjecture.

\begin{theorem}\label{thm:main_un_grps}
	Let $w\in\freegrp_k$ be a non-trivial word and $\varepsilon>0$.
	There exists an integer $N(\varepsilon,w)$ such that $w(\U_n)$ is $\varepsilon$-dense in $\U_n$ with respect to the normalized rank metric for all $n\geq N(\varepsilon,w)$.
\end{theorem}

First of all, note that in the metric context there is no notable difference between ${\rm A}_n$ and $\Sgrp_n$, $\GL_n(q)$ and $\SL_n(q)$, $\GO_{2n'+1}(q)$ resp.\ $\GO_{2n'}^\pm(q)$ and $\Omega_{2n'+1}$ resp.\ $\Omega^{\pm}_{2n'}(q)$, $\GU_n(q)$ and $\SU_n(q)$, and similarly between $\U_n$ and $\SU_n$ when $n$ is large, so that we are essentially talking about families of quasisimple compact groups. We believe that analogous results hold for other families of compact Lie groups of increasing rank. Also note that density with respect to the normalized rank metric implies density with respect to the normalized Hilbert--Schmidt metric on $\U_n$ -- this was also unknown to the best of our knowledge.

\vspace{0.1cm}

Let us now say some words about the proofs of Theorem~\ref{thm:main_sym_grps},~\ref{thm:main_thm_fin_grp_lt}, and~\ref{thm:main_un_grps}. First observe that it suffices to prove both results for $k=2$. Indeed, if $w\in\freegrp_k$ for $k\geq 2$, then, via a suitable embedding $\freegrp_k\subseteq\freegrp_2=\gensubgrp{x,y}$, we can view $w$ as a non-trivial word in the two variables $x,y$. Hence we shall restrict to the case $w\in\freegrp_2=\gensubgrp{x,y}$. We may also assume that $w$ is cyclically reduced, since $w(G)$ is a characteristic subset for any group $G$, so that, up to a change of variables, $w=x^a$ or $w=x^{a_1}y^{b_1}\cdots x^{a_ l}y^{b_l}$ for some $l\in\ints_+$.

Now let $(G,d_G)$ be one of the groups from Theorem~\ref{thm:main_sym_grps},~\ref{thm:main_thm_fin_grp_lt}, or~\ref{thm:main_un_grps} together with the corresponding metric. Write $n=n(G)$ for its degree.
In all three cases, instead of proving the existence of $N(w,\varepsilon)$, we will prove the equivalent statement (including the corresponding quantitative bounds for $N(w,\varepsilon)$ mentioned below) that there exists a function $d\colon [0,1]\to\reals$ of type $d(x)=Cx^{1/e}$, where $e=e(w)\geq 1$ only depends on $w$ and $C>0$ depends on the choice of $e$, such that $d_G(g,w(G))\leq d(1/n)$ for all $g\in G$. This just means that $N(w,\varepsilon)=O_w((1/\varepsilon)^{e(w)})$, where now the implied constant in the $O$ notation may still depend on $w$.

\vspace{0.1cm}

We start out by giving a brief outline of the proofs of Theorem~\ref{thm:main_sym_grps},~\ref{thm:main_thm_fin_grp_lt}, and \ref{thm:main_un_grps}. For the convenience of the reader, we will prove Theorem~\ref{thm:main_un_grps} before proving Theorem~\ref{thm:main_thm_fin_grp_lt}, as their proofs follow the same idea, but the details of the latter are more involved.

We start with Theorem~\ref{thm:main_sym_grps}. First of all, we need to fix some notation. We say that an element $\sigma\in \Sgrp_n$ is of \emph{cycle type} $(k^{c_k})_{k\in\ints_+}$ if it has exactly $c_k$ cycles of length $k$. Then $\sigma$ is called $k$-\emph{isotypic} if it has only cycles of length $k$, i.e., $n=c_k k$. We call $\sigma$ \emph{isotypic} if it is $k$-isotypic for some $k\in\ints_+$. 
To prove Theorem~\ref{thm:main_sym_grps} we first settle the case where $\sigma$ is isotypic using the cycle structure of elements of $\PSL_2(q)$ acting on the projective line $L_q$ (see Subsections~\ref{subsec:cyc_str_el_psl_2} and \ref{subsec:eff_surj_tr}), and then deduce the general case using an application of Jensen's inequality and the fact that a permutation $\sigma\in \Sgrp_n$ has less than $\sqrt{2n}$ distinct cycle types (see Subsection~\ref{subsec:prf_mn_thm_sym_grps}).
Note that the idea used in the proof of partitioning the set $\underline{n}$ into copies of projective lines $L_q$ and letting copies of groups $\PSL_2(q)$ act on them already appears in \cite[Proof of Proposition~8]{larsen2004word}.		
In this case, one needs a number theoretic result by Linnik \cite{linnik1944least} to prove the existence of the constant $e=e(w)\geq 1$. However, the qualitative statement of Theorem~\ref{thm:main_sym_grps} remains true without this assumption. Assuming a conjecture of Chowla \cite{chowla1934} one can show that any $e>2(l+1)$ works. 

We proceed by giving an overview of the proof of Theorem~\ref{thm:main_un_grps}. The proof of the results in \cite{elkasapythom2014goto} relied on the analysis of a certain algebraic condition on the abelianized Fox derivative of $w$ and our new strategy is a generalization of this -- see Subsection~\ref{subsec:con_rmk} for details. We use monomial matrices and draw a connection to the normalized dimension of the second cohomology group of finite quotients of the Cayley complex of the one-relator group $K=\freegrp_2/\gennorsubgrp{w}$. 
As an intermediate step, we consider the largest free nilpotent group $H$ which is a quotient of $K$, i.e., when $c=c(w)\geq 0$ is determined by $w\in\gamma_{c+1}(\freegrp_2)\setminus\gamma_{c+2}(\freegrp_2)$, where $(\gamma_i(L))_{i\geq 1}$ denotes the lower central series of the group $L$,
then $H=\freegrp_2/\gamma_{c+1}(\freegrp_2)$. Using Jennings' embedding theorem, for all sufficiently large primes $p$, we find arbitrary large finite $p$-groups $H(p)$ of composition length $h=h(w)$ equal to the Hirsch length of $H$, which are quotients of $K$ and where the above normalized dimension gets arbitrarily small. A quantitative analysis then reveals that one can take any number greater than $h(w)$ for the exponent $e=e(w)$. In the worst case we get that $c \leq 2l$ by a result of Fox \cite{fox1953free} and then $h(w)\leq\sum_{k=1}^c 2^k<2^{2l+1}$.

In Subsection~\ref{subsec:furth_impl}, we point out that our method of proof together with a fact on the linearized permutation representation of $\Sgrp_n$ (see Lemma~\ref{lem:bd_wdth_perm_rep}) implies as well that $w_1(\SU_n)w_2(\SU_n)=\SU_n$ for non-trivial words $w_1,w_2\in\freegrp_k$ and large $n$, providing an alternative proof for \cite[Theorem~2.3]{huilarsenshalev2015waring}. However, it still remains unclear how to prove surjectivity of single words $w$ in general.

Finally, in Section~\ref{sec:fin_grps_lt} we use the same cohomological method of Section~\ref{thm:main_un_grps} but with coefficient groups $k[X]/(\chi)^\times$ for $\chi\in k[X]$ a polynomial instead of the circle $\U_1$ and a modified version of Lemma~\ref{lem:bd_dim_rt_idl} (namely Corollary~\ref{cor:bd_dim_rt_idl_ints}) to settle Theorem~\ref{thm:main_thm_fin_grp_lt}. We remark here that our proof for $\GL_n(k)$ works for all fields $k$ (not only finite ones) and we conjecture that the same is true for the other Lie types.

After finishing a first version of this text, we noted that there is an alternative root to the proof of Theorem~\ref{thm:main_sym_grps}, which uses the ideas from Section~\ref{sec:un_grps} and~\ref{sec:fin_grps_lt}, but with finite cyclic groups $\Cycgrp_k$ instead of coefficients in $\U_1$. It can be found in Subsection~3.4.3 of \cite{schneider2019phd}.

\vspace{0.1cm}

We do not want to say much about metric ultraproducts. However, the statement that $w$ has dense word image on a suitable class of groups carrying bi-invariant metrics (as in Theorems~\ref{thm:main_sym_grps},~\ref{thm:main_thm_fin_grp_lt}, and~\ref{thm:main_un_grps}) is equivalent to surjectivity of $w$ on metric ultraproducts of those groups, where $n(G)\to\infty$ along the ultrafilter. 
Forming such a metric ultraproduct of symmetric groups leads to a so-called universal sofic group, whereas for complex unitary groups we obtain a group which surjects continuously on a universal hyperlinear group. For more on this subject, see Pestov's survey \cite{pestov2008hyperlinear} and the references therein or \cite{nikolovschneiderthom2017some, thomwilson2018some}.

\vspace{0.1cm}

Let us end by drawing some connections to related questions and other articles. To prove that the cardinality of the word image $w(G)$ for $G$ a quasisimple group is large (which was done in \cite[Theorem~2]{larsen2004word} and \cite[Theorems~1.9 and~1.11]{larsenshalev2009word}), Shalev and Larsen approximate an element which has a logarithmically large conjugacy class (e.g., in $\Sgrp_n$ an $n$-cycle and in a classical group of Lie type an element admitting a cyclic vector) and exploit that the cardinality of conjugacy classes is continuous in the normalized Hamming metric resp.\ the normalized rank metric, i.e., $\card{g^G}/\card{h^G}\leq\card{(gh^{-1})^G}$ which is bounded by $\card{G}^{Ld_{\rm H}(g,h)}$ resp.\ $\card{G}^{Ld_{\rk}(g,h)}$ for some constant $L>0$ (this is used, e.g., in \cite{liebeckshalev2001diameters}; see also~\cite[Corollary~2.14 and Theorem~2.15]{stolzthom2014lattice}). Hence metric density also implies that $\log_{\card{G}}\card{w(G)}\to 1$ when $n(G)\to\infty$. In private communication with Shalev, he conjectured, because of the above connection, that the word image $w(\Sgrp_n)$ is actually $C/n$-dense, for a fixed constant $C>0$, as indicated by \cite[Theorem~1.9]{larsenshalev2009word} stating that $\card{w(\Sgrp_n)}\geq n^{-4-\varepsilon}n!$ for $n$ sufficiently large. To prove the latter fact, it is enough to find one conjugacy class in the image $w(\Sgrp_n)$ which comes from an element being $C$-close to an $n$-cycle, as such a class has a centralizer of order polynomially bounded in $n$. However, in \cite{larsen1997often} it is shown that for power words the even better estimate $\card{w(\Sgrp_n)}\geq Cn^{-1}n!$ holds, but Proposition~\ref{prop:pwr_wrds_est} of Section~\ref{subsec:prf_mn_thm_sym_grps} demonstrates, that though the word image misses some large $\varepsilon$-balls, i.e., $\varepsilon=\Omega(1/\sqrt{n})$. Hence the word image can be distributed quite non-uniformly in the metric sense.

Similarly, using results of \cite{liebeckshalev2001diameters} or \cite{dowerkthom2019bounded}, if one can approximate a conjugacy class of large cardinality resp.\ norm, one gets bounded width of the word image. In any case, maybe both questions about cardinality and width of the word image have the same simple answer, namely that $w$ is eventually surjective (when $w$ is not a proper power in the case of finite simple groups).

The rest of this article is structured as follows. In Section~\ref{sec:sym_grps}, we give the proof of Theorem~\ref{thm:main_sym_grps}, Section~\ref{sec:un_grps} presents the proof of Theorem~\ref{thm:main_un_grps}, and in Section~\ref{sec:fin_grps_lt} we prove Theorem~\ref{thm:main_thm_fin_grp_lt}.

\section{Symmetric groups}\label{sec:sym_grps}

This section is devoted to the proof of Theorem~\ref{thm:main_sym_grps}. We start by collecting some basic facts about the groups $\PSL_2(q)$ for $q$ a prime power.

\subsection{Cycle structure of elements from $\PSL_2(q)$}\label{subsec:cyc_str_el_psl_2}

In this subsection, we recall some well-known facts about the cycle structure of the elements from $\PSL_2(q)\subseteq\Sym(L_q)$ acting on the projective line $L_q=\Proj^1(\finfield_q)$, where $q=p^e$ is a power of the prime $p$. The key observation, which we will exploit in Subsection~\ref{subsec:prf_mn_thm_sym_grps} to prove Theorem~\ref{thm:main_sym_grps}, is here that these elements are all almost isotypic.

Consider an element $g\in\SL_2(q)$ and write $\overline{g}\in\PSL_2(q)$ for the corresponding permutation on $L_q$. Then $g$ has two eigenvalues $\lambda,\lambda^{-1}\in\finfield_{q^2}$. Denote by $o\coloneqq\ord(\lambda)$ the multiplicative order of $\lambda$. We have the following trichotomy.

\emph{Case 1:} If $\lambda=\pm 1$, then $g$ has at least one eigenvector. If it has a second eigenvector not contained in the span of the first one, we have $g=\pm\id$, so $\overline{g}=\id_{L_q}$. If this is not the case, in a suitable basis
$$
g=\pm\begin{pmatrix}
1 & 1\\
0 & 1
\end{pmatrix},
$$
so $\overline{g}$ has precisely one fixed point $[1:0]$ and the remaining $q/p$ cycles are all of length $p$.

\emph{Case 2:}
In the case when $\lambda\in\finfield_q\setminus\set{\pm 1}$, we see that $g$ is diagonalizable over $\finfield_q$, whence $o=\ord(g)$ divides $q-1$. So choose coordinates such that $g=\diag(\lambda,\lambda^{-1})$. Then $\overline{g}$ has the two fixed points $[1\!:\!0],[0\!:\!1]$ on $L_q$ corresponding to the eigenvectors of $g$ over $\finfield_q$. Take any other point $x=[a\!:\!b]\in L_q$ (i.e., $ab\neq 0$). Then the orbit of $x$ under $\gensubgrp{\overline{g}}\subseteq\PSL_2(q)$ has length $k=o/2$ resp.\ $k=o$ when $o$ is even respectively odd. Namely, $x.\overline{g}^l=x$ is equivalent to $[\lambda^la\!:\!\lambda^{-l}b]=[\lambda^{2l}a\!:\!b]=[a\!:\!b]$, which is equivalent to $\lambda^{2l}=1$, meaning that $o/2\mid l$ for $o$ even and $o\mid l$ for $o$ odd.

\emph{Case 3:}
In the last case, $\lambda\in\finfield_{q^2}\setminus\finfield_q$. Since $(X-\lambda)(X-\lambda^{-1})=\chi_g(X)\in\finfield_q[X]$, it follows that $\lambda^{-1}=\lambda^q$ is the Galois conjugate to $\lambda$ in $\finfield_{q^2}$, so $o=\ord(g)$ divides $q+1$. 
Moreover, $\overline{g}$ has no fixed points as $g$ has no eigenvector over $\finfield_q$. However, embedding into $\SL_2(q^2)$ we can again assume $g=\diag(\lambda,\lambda^{-1})$. Then the same argument as above shows that all cycles of $\overline{g}$ have length $k=o/2$ resp.\ $k=o$ when $o$ is even resp.\ odd.

To summarize our observations, let us state the following corollary.

\begin{corollary}\label{cor:cyc_tp_dp_evals}
The cycle type of $\overline{g}$ acting on $L_q$ only depends on $o=\ord(\lambda)$ and $q$ if $o>2$. Namely, if $2<o$ and $o\mid q-1$ (Case~2), then it is $(1^2,(o/2)^{2(q-1)/o})$ resp.\ $(1^2,o^{(q-1)/o})$ when $o$ is even resp.\ odd, and if $2<o$ and $o \mid q+1$ (Case~3), it is $((o/2)^{2(q+1)/o})$ resp.\ $(o^{(q+1)/o})$ when $o$ is even resp.\ odd.
\end{corollary}

\subsection{Effective surjectivity of word maps over finite fields}\label{subsec:eff_surj_tr}

In this subsection, using the facts from Subsection~\ref{subsec:cyc_str_el_psl_2}, we demonstrate that permutations of certain cycle type are attained as $w$-values inside groups of type $\PSL_2(q)$, thus providing the crucial ingredient for the proof of Theorem~\ref{thm:main_sym_grps} in Subsection~\ref{subsec:prf_mn_thm_sym_grps}.

As in the previous subsection, let $q=p^e$ be a power of the prime $p$.
The map $\tr_w\colon \SL_2(\overline{\finfield}_q)^2\to\overline{\finfield}_q$ defined by $(g,h)\mapsto\tr(w(g,h))$ is surjective. Indeed, if $p$ is sufficiently large, this can be seen from the existence of trace polynomials and the theorem of Borel \cite{borel1983free} that the word map associated to $w$ on $\SL_2(\overline{\finfield}_q)$ is dominant -- but it follows also from direct inspection as explained below. 
Here we show surjectivity of $\tr_w$ for $p$ large enough but in an effective way. Throughout this subsection, assume that $w$ is not a power word (i.e., not of the form $x^a$ or $y^b$) and $p\nmid a_i,b_i$ for $i=1,\ldots,l$, where $w=x^{a_1}y^{b_1}\cdots x^{a_l}y^{b_l}$ as in the introduction.

\begin{lemma}\label{lem:eff_surj}
	For any $t\in\finfield_q$ there exist a positive integer $m\leq l$ and unipotent elements $g,h\in\SL_2(q^m)$ such that $\tr_w(g,h)=t$.	
\end{lemma}

\begin{proof}
	A classical result (going back to \cite{vogt1889sur}; see also \cite{frickeklein1965vorlesungen,horowitz1972characters,magnus1980rings,traina1980trace,goldman2009trace}) says that 
	$$
	\tr_w(g,h)=f(\tr(g),\tr(h),\tr(gh))
	$$ 
	is a polynomial in $\tr(g)$, $\tr(h)$, and $\tr(gh)$, where $$f(X,Y,Z)=f_l(X,Y)Z^l+\cdots+f_0(X,Y)\in\ints[X,Y,Z]$$ is uniquely determined and called the \emph{trace polynomial} of $w$. Now we define $g(U,V),h(U,V)\in\SL_2(\ints[U,V])$ by
	$$
	g(U,V)\coloneqq\begin{pmatrix}
	1 & 0\\
	U & 1
	\end{pmatrix} \text{ and }
	h(U,V)\coloneqq\begin{pmatrix}
	1 & V\\
	0 & 1
	\end{pmatrix}.
	$$
	Then $\tr(g(U,V))=\tr(h(U,V))=2$ and $\tr(gh)=UV+2$ in $\ints[U,V]$. Computing $f_l(2,2)$ from above using \cite[Theorem~1]{traina1980trace} gives $f_l(2,2)=a_1b_1\cdots a_lb_l$, in particular $f$ is non-trivial.
	
	Now, substituting $1$ for $V$ and reducing modulo $p$ gives a polynomial $r(U)\coloneqq\tr_w(g(U,1),h(U,1))=a_1b_1\ldots a_lb_lU^l+s(U)\in\finfield_p[U]$, with $\deg(s)<l$, of degree $l$ by assumption on $w$, as $p\nmid a_1b_1\cdots a_lb_l$. Hence the equation $r(U)-t=0$ is an equation over $\finfield_q$ of degree $l$, so has a solution in one of the fields $\finfield_{q^i}$ for $i=1,\ldots,l$. 
\end{proof}

\begin{remark}\label{rmk:odd_deg_ext}
	If $l$ is odd, then $m$ can also be chosen odd, since then at least one irreducible factor of $r(U)-t$ must be of odd degree.
\end{remark}

As a consequence of Lemma~\ref{lem:eff_surj} together with the facts mentioned in Subsection~\ref{subsec:cyc_str_el_psl_2}, for any fixed integer $k>1$, we get a word value in some groups of the form $\PSL_2(q^{im})$ for all $i\in\ints_+$, where $q$ depends on $k$, consisting only of $k$-cycles up to two fixed points. We conclude the following corollary.

\begin{corollary}\label{cor:tr_to_cyc_tp}
	Let $k>1$ be an integer. Assume that $2k\mid q-1$ resp.\ $k\mid q-1$ when $k$ is even resp.\ odd. Then there exists $m\leq l$ such that there is an element $\sigma\in w(\PSL_2(q^{im}))\subseteq \Sgrp_{q^{im}+1}$ of cycle type $(1^2,k^{(q^{im}-1)/k})$ for all $i\in\ints_+$.
\end{corollary}

\begin{proof}
	Let $i\in\ints_+$ be arbitrary. Choose $\lambda\in\finfield_q^\times$ of order $2k$ resp.\ $k$ when $k$ is even resp.\ odd. Then apply Lemma~\ref{lem:eff_surj} to $t\coloneqq\lambda+\lambda^{-1}\in\finfield_q$ to get $g,h\in\SL_2(q^m)\subseteq\SL_2(q^{im})$ for some $m\leq l$ with $\tr_w(g,h)=\tr(w(g,h))=t$. Note that $w(g,h)\in\SL_2(q^m)\subseteq\SL_2(q^{im})$ is diagonalizable with eigenvalues $\lambda,\lambda^{-1}\in\finfield_q^\times\subseteq\finfield_{q^{im}}^\times$.
	Setting $\sigma\coloneqq\overline{w(g,h)}\in\PSL_2(q^{im})\subseteq \Sgrp_{q^{im}+1}$, Corollary~\ref{cor:cyc_tp_dp_evals} of Subsection~\ref{subsec:cyc_str_el_psl_2} immediately implies the claim.
\end{proof}

\begin{remark}
	If $l$ is odd, using Remark~\ref{rmk:odd_deg_ext}, one can even remove the two fixed points from the above element $\sigma$. Indeed, assuming $2k\mid q+1$ resp.\ $k\mid q+1$ when $k$ is even resp.\ odd and going through the proof of Corollary~\ref{cor:tr_to_cyc_tp} together with the fact that $m$ can be chosen odd, one gets $\sigma\in\Sgrp_{q^{im}+1}$ of cycle type $(k^{(q^{im}+1)/k})$ for all odd $i\in\ints_+$.
\end{remark}

The next result shows that there is also a word value in $\PSL_2(q)\subseteq\Sym(L_q)\cong \Sgrp_{q+1}$, which is close to a $(q+1)$-cycle in $\Sgrp_{q+1}$. It can be considered as a weak version of \cite[Theorem~4.1]{larsenshalev2009word} which permits an elementary proof.

\begin{lemma}\label{lem:approx_lng_cyc}
	Assume that $q>4l$. Then there exists $\sigma\in w(\PSL_2(q))\subseteq\Sym(L_q)\cong \Sgrp_{q+1}$, which differs in less than $2+\sqrt{lq}$ points of $L_q$ from a $(q+1)$-cycle in $\Sgrp_{q+1}$.
\end{lemma}

\begin{proof}
	Using the same trick as in the proof of Lemma~\ref{lem:eff_surj}, one sees that the map $\tr_w\colon \SL_2(q) \times \SL_2(q)\to\finfield_q$ meets at least $q/l$ points (as $\tr_w(g,h)$ is a polynomial of degree $l$ in the traces $\tr(g)$, $\tr(h)$, and $\tr(gh)$). This implies that the set $\Lambda\subseteq\finfield_{q^2}^\times$ of eigenvalues of elements from $w(\SL_2(q))$ has cardinality at least $2(q/l-1)$ (two eigenvalues for each trace value apart from the trace values $\pm 2$; if $2\mid q$ one can take $2q/l-1$). Now assume that the multiplicative order of each of these eigenvalues is less than $b\coloneqq 2\sqrt{q/l}$. Then we obtain the contradiction
	$$
	\card{\Lambda}\leq\sum_{i=1}^{\ceil{b}-1}{\varphi(i)}\leq \frac{(\ceil{b}-1)^2}{2}-2<2(q/l-1),
	$$
	where the second inequality holds since $\varphi(i)\leq i-1$ and by assumption $\ceil{b}-1\geq 4$.
	Hence $\Lambda$ contains an element $\lambda$ of order $o\geq 2\sqrt{q/l}>4$. Let $f\in w(\SL_2(q))$ with eigenvalues $\lambda,\lambda^{-1}\in\finfield_{q^2}$. Then by Corollary~\ref{cor:cyc_tp_dp_evals}, $\sigma\coloneqq\overline{f}\in\PSL_2(q)$ consists apart from zero or two fixed points only of cycles of type $o/2$ resp.\ $o$ when $o$ is even resp.\ odd. This implies that $\sigma$ differs in less than $2+\sqrt{lq}$ points from a $(q+1)$-cycle in $\Sym(L_q)\cong \Sgrp_{q+1}$.
\end{proof}

\begin{remark}
	For even $q$, one can improve the estimate by a factor $1/2$, since $o$ will always be odd.
\end{remark}

\subsection{Proof of Theorem~\ref{thm:main_sym_grps}}\label{subsec:prf_mn_thm_sym_grps}

In this subsection, we use the facts provided by Subsection~\ref{subsec:eff_surj_tr} to establish Theorem~\ref{thm:main_sym_grps}. Hence let us first assume that $w=x^{a_1}y^{b_1}\cdots x^{a_l}y^{b_l}$ as in Subsection~\ref{subsec:eff_surj_tr} and postpone the case that $w$ is a power word to the end of this subsection.
We start of with the isotypic case and prove the general case as a consequence.

\emph{The isotypic case.} At first let $\sigma\in \Sgrp_n$ be $k$-isotypic, i.e., $n=c_k k$. We can certainly restrict to $k>1$, since the identity is always in $w(\Sgrp_n)$. Subsequently, we prove two estimates for the quantity $d_{\rm H}(\sigma,w(\Sgrp_n))$. The first estimate will be suitable for small $k$, whereas the second will be better for large $k$.

\emph{Estimate for small $k$.}
Let $p$ be the smallest prime such that $p\nmid a_i,b_i$ for $i=1,\ldots,l$ and $2k\mid p-1$ resp.\ $k\mid p-1$ when $k$ is even resp.\ odd. Apply Corollary~\ref{cor:tr_to_cyc_tp} to $q\coloneqq p$ to get the integer $m\leq l$. Set $q\coloneqq p^m$ and write 
$$
n=\sum_{i=1}^s{n_i(q^i+1)}+n_0
$$ 
so that $\sum_{i=1}^j{n_i(q^i+1)}+n_0\leq q^{j+1}$ for all $0\leq j\leq s$ (i.e., use a greedy algorithm to produce such a representation, starting with the biggest summand $q^s+1$). 

Then $n_i\leq q-1$ for $i\geq 1$ and $n_0\leq q$. Moreover, using a standard estimate for the $q$-ary representation of positive integers, one obtains $\sum_{i=0}^s{n_i}<q(\log_q(n)+1)$.

Write $\underline{n}=\bigsqcup_{i=1}^s{n_iL_{q^i}}\sqcup\underline{n_0}$ as a disjoint union of $n_i$ copies of the projective lines $L_{q^i}$ for $i=1,\ldots,s$ and $n_0$ singletons. Using Corollary~\ref{cor:tr_to_cyc_tp}, let $g,h\in \Sgrp_n$ be permutations, which restrict to maps $\overline{g},\overline{h}\in\PSL_2(q^i)$ acting on the copies of $L_{q^i}$ so that $w(\overline{g},\overline{h})$ has cycle type $(1^2,k^{(q^i-1)/k})$, and which fix the remaining $n_0$ points. Then, if we label the points in an optimal way, we get $$nd_{\rm H}(\sigma,w(g,h))=n_0+2\sum_{i=1}^m{n_i}\leq 2\sum_{i=0}^m{n_i}\leq 2q(\log_q(n)+1).$$

By a celebrated result of Linnik \cite{linnik1944least}, one has that the least prime which is congruent to $1$ modulo $2k$ resp.\ $k$ is bounded by $D_1k^{D_2}$ for some constants $D_1>0$, $D_2\geq 1$. Choosing $D_1$ large enough, one can also ensure that $p\nmid a_i,b_i$ for $i=1,\ldots,l$, e.g., take $p$ congruent to $1$ modulo $2ka_1b_1\cdots a_lb_l$. Hence $q\leq D_1^l k^{D_2 l}$, so that
$$
d_{\rm H}(\sigma,w(g,h))\leq 2D_1^l k^{D_2 l}(\log_2(n)+1)/n.
$$

\begin{remark}
	The logarithmic term in this argument can be removed if $l$ is odd. Namely, then we require $2k\mid p+1$ resp.\ $k\mid p+1$ for $k$ even resp.\ odd, and can choose $m\leq l$ odd, so that we are in Case~3 of Subsection~\ref{subsec:cyc_str_el_psl_2}, where no fixed points occur. However, then we may only use the odd $i$ and thus get a bigger constant.
	
	It is probably also true that, when $w$ is not a square, $\sigma\in w(\Sgrp_n)$ if $k$ is fixed and $c_k$ is even and large enough in terms of $k$ (Lemma~3.42(ii) in Subsection~3.4.3 of \cite{schneider2019phd} can be seen as a weak form of this conjecture which is true). But this also would not improve our estimate.
	
	The result of Linnik is not necessary for the qualitative statement of Theorem~\ref{thm:main_sym_grps}. We only need it to get a nice function $d$, which is mentioned in the introduction.
	There are weaker versions of Linnik's result available with an elementary proof, e.g., see \cite{thangaduraivatwani2011least}.
\end{remark}

\emph{Estimate for large $k$.}
Let $p$ be the smallest prime such that $p\nmid a_i,b_i$ for $i=1,\ldots,l$ and $p>4l$. Set $q\coloneqq p$ and write $n=\sum_{i=1}^s{n_i(q^i+1)}+n_0$ and $\underline{n}=\bigsqcup_{i=1}^s{n_iL_{q^i}}\sqcup\underline{n_0}$ as above. Using Lemma~\ref{lem:approx_lng_cyc}, let $g,h\in \Sgrp_n=\Sym(\underline{n})$ be permutations, which restrict to maps $\overline{g},\overline{h}\in\PSL_2(q^i)$ acting on the copies of $L_{q^i}$ so that $w(\overline{g},\overline{h})$ differs in less than $2+\sqrt{lq^i}$ points from a $(q^i+1)$-cycle for $i=1,\ldots,s$, and which fix the remaining $n_0$ points.

Again, under an optimal labeling, an $n$-cycle differs from $\sigma$ in at most $c_k$ points. Hence, using the triangle inequality
$$
d_{\rm H}(\sigma,w(\Sgrp_n))\leq \frac{1}{k}+\frac{1}{n}\left(\sum_{i=1}^s{n_i(\sqrt{lq^i}+2)}+n_0\right).
$$
The second term can be estimated by $D_3/\sqrt{n}$ for suitable $D_3>0$ depending on $q$ and $l$.

\begin{remark}
	By \cite[Theorem~1.3]{larsenshalev2009word} there exists a constant $D_4>0$ such that there are elements $g,h\in \Sgrp_n$ which restrict to permutations on the support of each $k$-cycle of $\sigma$ such that $d_{\rm H}(\sigma,w(g,h))\leq D_4/k$. However, the proof presented there uses some algebraic geometry and the weak Goldbach conjecture, and using it instead of the above estimate would not improve the exponent $e$ mentioned in the introduction. Note here, that we found an alternative proof of this result of Larsen and Shalev after having finished a first version of this article, which is presented in Lemma~3.42(iii) of \cite{schneider2019phd}.
\end{remark}

\emph{Global estimate for isotypic elements.}
Using the first estimate for
$$
k\leq\left(\frac{n}{2D_1^l(\log_2(n)+1)}\right)^{\frac{1}{D_2l+1}}
$$
and the second one in the opposite case, one obtains that 
$$
d_{\rm H}(\sigma,w(\Sgrp_n))\leq C_{\rm it} n^{-1/e_{\rm it}}
$$
for any $e_{\rm it}>D_2l+1\geq 2$ and $C_{\rm it}$ appropriately. 
Assuming a conjecture by Chowla~\cite{chowla1934}, we can take $D_2$ arbitrarily close to one, so that $e_{\rm it}$ can be taken arbitrarily close to $l+1$.

\vspace{0.1cm}

We will now use our knowledge about the isotypic case to conclude the proof of the theorem in the general case.

\emph{The general case.} Now we are ready to establish Theorem~\ref{thm:main_sym_grps}. 
A basic ingredient we need is the elementary fact that a permutation on $n$ letters has less than $\sqrt{2n}$ different cycle types.

\begin{proof}[Proof of Theorem~\ref{thm:main_sym_grps}]
	Let us restrict to the case when $w$ is not a power word, so that we can use the above estimates. The opposite case is clarified below.
	
	Set $d_{\rm it}(x)\coloneqq C_{\rm it} x^{1/e_{\rm it}}$ and note that $d_{\rm it}$ is monotone and concave.
	Let $\Omega_k\coloneqq\Omega_k(\sigma)$ be the support of all $k$-cycles of $\sigma\in \Sgrp_n$ and $n_k\coloneqq\card{\Omega_k}$ for $k\in\ints_+$. Let $S$ be the set of positive integers $k$, such that $n_k >0$ and note that $\card{S}<\sqrt{2n}$. Then for $n\geq 2$
	\begin{align*}
	d_{\rm H}(\sigma,w(\Sgrp_n)) &\leq\sum_{k\geq 1}{ \frac{n_k}{n}d_{\rm H}(\rest{\sigma}_{\Omega_k}\!,w(\Sym(\Omega_k)))}\\
	&\leq\sum_{k \in S}{\frac{n_k}{n}d_{\rm it}\left(\frac{1}{n_k}\right)}\\	
	&\leq d_{\rm it}\left(\sum_{k\in S}{\frac{1}{n}}\right)\leq d_{\rm it}\left(\frac{\sqrt{2n}}{n}\right)=d_{\rm it}\left(\sqrt{2/n}\right),
	\end{align*}
	where the second last inequality is implied by Jensen's inequality applied to the concave function $d_{\rm it}$, and the last one by monotonicity of $d_{\rm it}$. We can now set $d(x)\coloneqq d_{\rm it}(\sqrt{2x})=\sqrt{2}C_{\rm it} x^{1/(2e_{\rm it})}$. This finishes the proof.
\end{proof}

	For power words $w=x^a$ with $a>1$ one can show that 
	$$
	d_{\rm H}(\sigma,w(\Sgrp_n))\leq a/n
	$$ 
	for all $n\in\ints_+$ and $\sigma\in \Sgrp_n$ isotypic with equality for infinitely many $n$ and suitable $\sigma$ (cf.\ Lemma~3.4 of \cite{schneider2019phd}). Hence, in view of Proposition~\ref{prop:pwr_wrds_est} below, the argument from above using Jensen's inequality produces the optimal bound $N(w,\varepsilon)=O((1/\varepsilon)^2)$ in this case. To state this proposition, we need the following terminology.
	For $x,y\in\ints_+$, define the \emph{$y$-part} $\pi_y(x)$ of $x$ as largest divisor of $x$ which is a product of powers of primes which divide $y$.

\begin{proposition}[cf.\ Lemma~3.5 and 3.6 of \cite{schneider2019phd}]\label{prop:pwr_wrds_est}
	Let $w=x^a$ for $a>1$ and $\varepsilon>0$. 
	Set 
	$$
	D=\frac{1}{a}\sum_{i=1}^a{(\pi_k(i)-1)}.
	$$
	We have that
	$$
	 (1-\varepsilon) \sqrt{\frac{D}{4n}} <\sup_{\sigma\in \Sgrp_n}d_{\rm H}(\sigma,w(\Sgrp_n))< (1+\varepsilon) \sqrt{\frac{2D}{n}}
	$$ for $n$ sufficiently large.
\end{proposition}

Since the argument is quite lengthy and rather straightforward, we omit it here and refer to \cite{schneider2019phd} for the proof.

\section{Unitary groups}\label{sec:un_grps}

In this section, we present the proof of Theorem~\ref{thm:main_un_grps} (Subsection~\ref{subsec:prf_thm_2} below) and draw some connections to the previous article~\cite{elkasapythom2014goto} of Elkasapy and the second author (see Subsection~\ref{subsec:con_rmk}).
	
\subsection{Proof of Theorem~\ref{thm:main_un_grps}}\label{subsec:prf_thm_2}

Denote by $K$ the one-relator group 
$$
\grp{x,y}[w]=\freegrp_2/\gennorsubgrp{w}
$$ 
associated to $w$.

The key observation is the following lemma involving the second cohomology group of a quotient of the Cayley complex of $K$, in which we interpret monomial matrices in the complex unitary group $\U_n$ as $1$-cochains.

Let $X$ be the Cayley complex of the presentation $\grp{x,y}[w]$ of $K$, i.e., its $1$-skeleton is the directed Cayley graph $\Gamma=\Cay(K,\set{x,y})$ and for each vertex $v\in V(\Gamma)=K$ we glue in a $2$-cell $c_v$ along $w$ starting at $v$. 
	
For $\pi\colon K\twoheadrightarrow G$ be a surjective homomorphism to a finite group $G$ of order $n$, set $g\coloneqq\pi(x)$, $h\coloneqq\pi(y)$ and let $X(\pi)$ be the quotient of the $2$-complex $X$ induced by $\pi$, whose $1$-skeleton is the Cayley graph $\Gamma(\pi)=\Cay(G,\set{g,h})$ of $G$. Consider also permutations $\sigma_g,\sigma_h \in \Sgrp_n$ arising from the action of $G$ on itself. Set $d(\pi)\coloneqq \dim H^2(X(\pi),\reals)$ to be the dimension of the second cohomology group of $X(\pi)$.

\begin{lemma}\label{lem:cohom_lem}
	For every diagonal unitary matrix $M\in\U_n$, we can find monomial matrices $M_g,M_h\in\U_1\wr \Sgrp_n\subseteq\U_n$ such that $M_g=(\lambda_i)_{i=1}^n.\sigma_g$ and $M_h=(\mu_i)_{i=1}^n.\sigma_h$, such that $w(M_g,M_h)$ is diagonal and differs in at most $d(\pi)$ diagonal entries from $M$. Hence, setting $\varepsilon(\pi)\coloneqq d(\pi)/n$, the image of the word map $w(\U_n)$ is $\varepsilon(\pi)$-dense in $\U_n$. 
\end{lemma}

\begin{proof}
	Write $C_\bullet(\pi)$ resp.\ $C^\bullet(\pi)$ for the chain resp.\ cochain complex over $\reals$ associated to $X(\pi)$ with differentials $d_i\colon C_i\to C_{i-1}$ resp.\ codifferentials $d^i\coloneqq d_i^\ast\colon C^{i-1}\to C^i$ ($i\in\nats$).
	A $1$-cochain $\alpha\colon X_1(\pi)\to\reals$ assigns to each edge $e$ of $\Gamma(\pi)$ a real number $\alpha_e$. Then the Cayley graph $\Gamma(\pi)$ together with this assignment encodes two elements $g_\alpha,h_\alpha\in\reals\wr \Sgrp_n\subseteq\reals^{n\times n}$, where the permutation part of $g_\alpha$ resp.\ $h_\alpha$ is given by the action of $g$ resp.\ $h$ on the vertices $V(\Gamma(\pi))=G$ of $\Gamma(\pi)$, and the first part is induced by the values $\alpha_e$ ($e\in E(\Gamma(\pi))$).
	The group $\reals\wr \Sgrp_n$ can be seen as the set of monomial matrices in $\reals^{n\times n}$, where the entries marked along the corresponding permutations are added instead multiplied.
	
	Given the $1$-cochain $\alpha$, its image under the codifferential $d^2(\pi)\colon C^1(\pi)\to C^2(\pi)$ is defined by 
	$$
	d^2(\pi)(\alpha)(c)=\sum_{e\in\partial(c)} \varepsilon_e\alpha(e),
	$$ 
	for all $c\in X_2(\pi)$, where $\partial(c)$ is the set of edges of the boundary of the cell $c$ and $\varepsilon_e\in\set{\pm 1}$ is the corresponding orientation.
	Now $C^2(\pi)/\im(d^2(\pi))=H^2(X(\pi),\reals)$.
	
	Choose $M=\diag(\lambda_v)_{v\in G}\in\U(\ell^2G)=\U_n$ arbitrarily and find $\beta_v\in\reals$ such that $\lambda_v=e^{i\beta_v}$ for $v\in G$. Then there exists a function $\alpha\colon X_1(\pi)\to\reals$ such that $d^2(\pi)(\alpha)(c_v)=\beta_v$ for all but at most $d(\pi)$ vertices $v\in V(\Gamma(\pi))=G$.
	
	But, letting $\varphi\colon\reals\wr \Sgrp_n\to\U^1\wr \Sgrp_n\subseteq\U_n$ be the homomorphism induced by exponentiation, we also see that 
	$$
	w(\varphi(g_\alpha),\varphi(h_\alpha))=(e^{i d^2(\pi)(\alpha)(c_v)})_{v\in G}.\id.
	$$ 
	Hence $M_g\coloneqq\varphi(g_\alpha)$, $M_h\coloneqq\varphi(h_\alpha)$ is a suitable choice of matrices. The last statement of the lemma follows from the definition of the normalized rank metric on $\U_n$. This completes the proof.
\end{proof}

\begin{remark}\label{rmk:surj_wrd_mp}
	Subsequently, for a chain $x\in C_i(\pi)$ ($i=0,1,2$) write $x^\ast\in C^i(\pi)$ for the corresponding dual cochain defined by $\scprod{x,\cdot}=x^\ast$, where $\scprod{\cdot,\cdot}$ is the inner product associated to the basis $X_i(\pi)$ of $C_i(\pi)$.
	
	If $w\in\freegrp_2'$, it is clear that $d(\pi)\geq 1$, as any element 
	$$
	\sum_{v\in G}{\lambda_v c_v^\ast}\in\im(d^2(\pi))
	$$ 
	lies in the hyperplane given by $\sum_{v\in G}{\lambda_v}=0$ (here we use that $G$ is finite). This reflects the fact that then $w(\U_n)\subseteq\SU_n$.
	Moreover, in this case, if $d(\pi)=1$, the word map $w\colon\SU_n\times\SU_n\to\SU_n$ is surjective (by transitivity we can then achieve equality on any $n-1$ diagonal entries in the above proof). Namely, if $w(g,h)=u$ for $g,h\in\U_n$ and $u\in\SU_n$, we can find $\lambda,\mu\in\complex$ such that $\lambda^n=\det(g)$ and $\mu^n=\det(h)$. Then $g'\coloneqq\lambda^{-1}g$, $h'\coloneqq\mu^{-1}h$ lie in $\SU_n$ and satisfy $w(g',h')=u$.
	
	In the opposite case, when $w\in\freegrp_2\setminus\freegrp_2'$, either $w(1,x)$ or $w(x,1)$ is of the form $x^m$ for $m\in\ints\setminus\set{0}$. So the word $w$ is always surjective on $\U_n$ and $\SU_n$ since every element of these groups is diagonalizable and hence has an $m$th root (of determinant one in case of $\SU_n$).
\end{remark}

\begin{remark}
	The map $d^2(\pi)$ of the Lemma~\ref{lem:cohom_lem} also makes sense for a surjective homomorphism $\pi\colon K\twoheadrightarrow G$ onto an infinite group $G$. We will consider such a case (namely for the group $G=H$ defined later; see also Lemma~\ref{lem:d_1*_nontriv}).
\end{remark}

\begin{remark}\label{rmk:fox_der}
	In the situation of the proof of Lemma~\ref{lem:cohom_lem}, write $C_\bullet=C_\bullet(X)$ resp.\ $C^\bullet=C^\bullet(X)$ for the chain resp.\ cochain complex over $\reals$ associated to $X$.
	Then we have a commutative diagram
	$$
	\begin{tikzcd}
	0 \arrow{r} & C^0 \arrow{r}{d^1}\arrow{d} & C^1 \arrow{r}{d^2}\arrow{d} & C^2 \arrow{r}\arrow{d} & 0\\
	0 \arrow{r} & C^0(\pi) \arrow{r}{d^1(\pi)} & C^1(\pi) \arrow{r}{d^2(\pi)}  & C^2(\pi) \arrow{r} & 0
	\end{tikzcd},
	$$
	where the top arrows are $K$-equivariant, the bottom arrows are $G$-equivariant, and the vertical arrows are induced by $\pi$. The duality defined in Remark~\ref{rmk:surj_wrd_mp} identifies $C^i(\pi)$ $G$-equivariantly with $C_i(\pi)$. Then identifying via the isomorphisms 
	$$
	C_1(\pi)\cong g\cdot\reals[G]\oplus h\cdot\reals[G]\quad\text{ and }\quad C_2(\pi)\cong\reals[G],
	$$ 
	where the latter is given by $c_v\mapsto v$ ($v\in G$), and letting $^\ast\colon\reals[G]\to\reals[G]$ be the natural involution induced by $g\mapsto g^{-1}$, one computes that the map $d^2(\pi)$ is then given by $d^2(\pi)(g\cdot 1)=\pi(\partial w/\partial x)^\ast$ and $d^2(\pi)(h\cdot 1)=\pi(\partial w/\partial y)^\ast$. Here, $\partial w/\partial x$ resp.\ $\partial w/\partial y$ denote the Fox derivative \cite{fox1953free} of $w$ with respect to $x$ resp.\ $y$, i.e., if $w=x^{a_1}y^{b_1}\cdots x^{a_l}y^{a_l}$ and $\varepsilon_i\coloneqq\sgn(a_i)$, $\delta_i\coloneqq\sgn(b_i)$, then
	\begin{align*}
	\frac{\partial w}{\partial x} &=\sum_{i=1}^l{\varepsilon_i x^{a_1} y^{b_1}\cdots x^{a_{i-1}}y^{b_{i-1}}x^{\frac{\varepsilon_i-1}{2}}(1+x^{\varepsilon_i}\cdots+x^{\varepsilon_i\abs{a_i-1}})},\\
	\frac{\partial w}{\partial y} &=\sum_{i=1}^l{\delta_i x^{a_1} y^{b_1}\cdots x^{a_i}y^{\frac{\delta_i-1}{2}}(1+y^{\delta_i}\cdots+x^{\delta_i\abs{b_i-1}})}.
	\end{align*}	
\end{remark}

Later we will apply Lemma~\ref{lem:cohom_lem} to a family of surjective homomorphisms $\pi(p)\colon K\twoheadrightarrow H(p)$ ($p$ a sufficiently large prime) to finite groups $H(p)$ of order $n_p=p^h$ ($h$ is a constant defined later), so that $\varepsilon(\pi(p))\to 0$ as $p\to\infty$. This is only possible if the corresponding map $d^2(\pi(p))$ in the above proof for $G=H(p)$ is non-trivial for sufficiently large $p$. Hence, next we characterize when this happens for an arbitrary homomorphism $\pi\colon K\twoheadrightarrow G$.

\begin{lemma}\label{lem:d_1*_nontriv}
	Let $G=\freegrp_2/N=\gensubgrp{x,y}/N$ be a (not necessarily finite) quotient of the one-relator group $K$, i.e., $w\in N$, and set $g\coloneqq\overline{x}$, $h\coloneqq\overline{y}$ in $G$. Define $\pi\colon K\twoheadrightarrow G$ as in Lemma~\ref{lem:cohom_lem}. Then $d^2(\pi)$ in the proof of Lemma~\ref{lem:cohom_lem} is identically zero if and only if $w\in N'=[N,N]$.
\end{lemma}

\begin{proof}
	By assumption, we have $w\in N$. If $w\in N'$, then $w=\prod_{i=1}^k{n_i}$ is a product of elements $n_i\in N$ ($i=1,\ldots,k$), where the multiset $(n_i)_{i=1}^k$ equals $(n_i^{-1})_{i=1}^k$. Consider the $w$-loop $l_v(w)$ with arbitrary starting vertex $v\in V(\Gamma(\pi))$. The above shows that each subloop of $l_v(w)$ associated to $n_i$ ($i=1,\ldots,k$) returns to $v$ and hence any edge in $\Gamma(\pi)$ is traversed equally often in both directions. But then one sees immediately that $d^2(\pi)=0$.
	
	Conversely, the assumption $d^2(\pi)=0$ implies that the loops $l_v(w)$ $(v\in V(\Gamma(\pi)))$ traverse all of its edges equally often in both directions. Let $\Delta$ be the undirected simple graph which is the image of the loop $l_v(w)$. Then $\Delta$ is homotopic to a bouquet of circles each of which is traversed equally often in both directions by the loop $l$ corresponding to $l_v(w)$ under the chosen homotopy. But this means precisely that the homotopy class of $l$ lies in $\pi_1(\Delta)'$. Pulling back the generators of the group $\pi_1(\Delta)$ to elements of $N$, we see that $w\in N'$.
\end{proof}

Now we show how to define the maps $\pi(p)\colon K\twoheadrightarrow H(p)$ and quotients $H(p)$ appropriately  (for $p$ a sufficiently large prime) such that $\varepsilon(\pi(p))\to 0$ for $p\to\infty$.

Since $\freegrp_2$ is residually nilpotent, there exists a unique integer $c=c(w)\geq 0$ such that $w\in\gamma_{c+1}(\freegrp_2)\setminus\gamma_{c+2}(\freegrp_2)$. Set $H\coloneqq\freegrp_2/\gamma_{c+1}(\freegrp_2)$ to be the free $2$-generated nilpotent group of class $c$ (in which $w$ is trivial) and let $\pi\colon K\twoheadrightarrow H$ be the corresponding quotient map.

By Jennings' embedding theorem, every finitely generated torsion-free nilpotent group $N$ can be embedded into the group $U\coloneqq\UT_d(\ints)$ of upper triangular matrices over $\ints$ (such an embedding can even be explicitly computed from a polycyclic representation of $N$ by an algorithm due to Nickel \cite{nickel2006matrix}; see also \cite{gulweiss2017dimension} and \cite{segal2005polycyclic}). Since the factors of the lower central series $\gamma_i(\freegrp_2)/\gamma_{i+1}(\freegrp_2)$ ($i=1,\ldots,c$) are free abelian, $H$ is a poly-$\ints$ group and we obtain that $H$ can be concretely realized as a subgroup of $\UT_d(\ints)$ for some dimension $d=d(w)$. 

Define the central series $H_i\coloneqq H\cap\gamma_i(U)$ of $H$ for $i=1,\ldots,d$ (note that the group $\gamma_l(U)$ consists of the upper triangular matrices $u=(u_{ij})\in U$ with $u_{ij}=0$ for $1\leq i-j\leq l-1$). Then $H_i/H_{i+1}\leq\gamma_i(U)/\gamma_{i+1}(U)\cong \ints^{d-i}$ (the $i$th off-diagonal). Let $B_i\subseteq\ints^{d-i}$ be a basis of $H_i/H_{i+1}$ and set $h_i\coloneqq\dim(H_i/H_{i+1})=\card{B_i}$ ($i=1,\ldots,d-1$). Let $p$ be a prime not dividing some $h_i\times h_i$ minor of the $(d-i)\times h_i$-matrix associated to $B_i$ for all $i=1,\ldots,d-1$. Then, writing $U(p)\coloneqq\UT_d(\ints/(p))$ and letting $H(p),H_i(p)$ ($i=1,\ldots,d$) be the image of $H,H_i$ in $U(p)$, we see that $H_i/H_{i+1}\cong\ints^{h_i}\twoheadrightarrow H_i(p)/H_{i+1}(p)\cong(\ints/(p))^{h_i}$. 
Define $\pi(p)\colon K\twoheadrightarrow H(p)$ to be the induced quotient map to the finite $p$-group $H(p)$.

Now refine the central series $(H_i)_{i=1}^d$ to a central series $(L_j)_{j=1}^{h+1}$ such that $L_j/L_{j+1}\cong\ints$ for $j=1,\ldots,h$, where $h \coloneqq\sum_{i=1}^{d-1}{h_i}$ is the Hirsch length of $H$. Then still $L_j/L_{j+1}\cong\ints\twoheadrightarrow L_j(p)/L_{j+1}(p)\cong\ints/(p)$ for $j=1,\ldots,h$ and $p$ as above. Let $x_j\in H$ be such that $\gensubgrp{x_j}L_{j+1}=L_j$ for $j=1,\ldots,h$. Note that the map $d^2(\pi)$ associated to the surjective homomorphism $\pi\colon K\twoheadrightarrow H$ is non-trivial, since if it were trivial, then by Lemma~\ref{lem:d_1*_nontriv} applied to $\pi$ and $N=\ker(\pi)=\gamma_{c+1}(\freegrp_2)$ we would have $w\in N'=[\gamma_{c+1}(\freegrp_2),\gamma_{c+1}(\freegrp_2)]\subseteq\gamma_{c+2}(\freegrp_2)$, which is not the case by the choice of $c$. Hence, from the local nature of the definition of $d^2(\pi)$, it follows that there is an edge $e\in E(\Gamma(\pi))$ such that $0\neq d^2(\pi)(e^\ast)=\sum_v{\lambda_v c_v^\ast}$ with $\lambda_v\in\ints\setminus\set{0}$. This element corresponds to the element $0\neq y=\sum_v{\lambda_v v}\in\ints[H]$ in the group ring. Subsequently, let $k$ be a finite field of large enough characteristic such that the image of $y$ in $k[H]$ is non-trivial, and for $z\in k[H]$ an element in the group algebra of $H$, write $z(p)$ for its image in $k[H(p)]$ under the reduction map. It follows that for $p$ large enough, the elements $v$ in the support $\supp(y)\subseteq H$ are mapped injectively to the elements $v(p)=\pi(p)(v)\in\supp(y(p))\subseteq H(p)$ (e.g., take $p$ larger than all matrix entries of elements $v$ from $\supp(y)$).

Now the $k$-dimension of $\im(d^2(\pi(p)))$ can be bounded from below by the $k$-dimension of the right ideal $y(p)k[H(p)]$ as the action of $H(p)$ on the $2$-cells of $C(\pi(p))$ equals its right action on the group algebra $k[H(p)]$. We bound the dimension of the latter from below by the following lemma.

\begin{lemma}\label{lem:bd_dim_rt_idl}
	In this situation, the right ideal $y(p)k[H(p)]\subseteq k[H(p)]$ generated by $y(p)$ has $k$-dimension at least $(p-f)^h$ for a constant $f=f(w)$ only depending on $w$.
\end{lemma}

\begin{proof}
	For $l=0,\ldots,h$ write 
	$$
	y=\sum_{e\in\ints^{h-l}}{x_1^{e_1}\cdots x_{h-l}^{e_{h-l}}c_e}
	$$ 
	for $e=(e_1,\ldots,e_{h-l})$ and $c_e\in\ints[L_{h+1-l}]$.
	We prove by induction on $l$ that for all $e\in\ints^{h-l}$ the right ideal $c_e(p)k[L_{h+1-l}(p)]$ is either zero or has $k$-dimension at least $(p-f)^l$, obtaining the claim for $l=h$ as $c_{\emptyset}=y\neq 0$ and so $y(p)\neq 0$ for $p$ large enough.
	
	For $l=0$ there is nothing to prove, as $c_e$ is either zero or it spans a one-dimensional ideal in $k=k[L_{h+1}]$. Now for the induction step assume the statement is proven for $l\geq 0$. Let $e\in\ints^{h-1-l}$ be arbitrary and write $c_e=\sum_{i\in\ints}{x_{h-l}^i c_{(e,i)}}$. If $c_e=0$, we are done, so assume the opposite. Then, certainly, the set $S\coloneqq\set{i\in\ints}[c_{(e,i)}\neq 0]\neq\emptyset$ is an invariant of $y$ and so of $w$, hence $m\coloneqq\max S-\min S\leq f(w)$ for some function $f$ of $w$. Set $z\coloneqq x_{h-l}^{\min S}$. Since the right $k[L_{h-l}(p)]$-ideals generated by $c_e(p)$ and $(z^{-1}c_e)(p)$ have the same dimension, we may consider the element $u\coloneqq z^{-1}c_e$ instead of $c_e$. This equals $u=\sum_{i=0}^m{x_{h-l}^ic_{(e,i+\min S)}}$. Now it is easy to see that the set of linear combinations $\sum_{k=0}^{p-m-1}{(ux_{h-l}^kd_k)(p)}$ with $d_k\in k[L_{h+1-l}]$ arbitrary for $k=0,\ldots,p-m-1$ generate a $k$-subspace of dimension at least $(p-m)(p-f)^l$. Indeed, by choosing $d_0$ appropriately, one can obtain any element of $c_{(e,\min S)}k[L_{h+1-l}]$ as the left coefficient in $k[L_{h+1-l}]$ of $x_{h-l}^0$. Then, choosing $d_1$ such that $x_{h-l} d_1x_{h-l}^{-1}\in k[L_{h+1-l}]$ is appropriate, one can obtain any left coefficient in front of $x_{h-l}^1$ in some coset of the right ideal $c_{(e,\min S)}k[L_{h+1-l}]$, etc. 
	Since by assumption $p-m\geq p-f$, we are done.
\end{proof}

As a consequence of Lemma~\ref{lem:bd_dim_rt_idl}, we obtain the following immediate corollary.

\begin{corollary}\label{cor:est_p_grps}
	Applying Lemma~\ref{lem:cohom_lem} to $\pi(p)$ as above, we obtain that $\varepsilon(\pi(p))\leq 1-(1-f/p)^h\leq hf/p=hf n^{-1/h}$ for $n=p^h$, where $h=h(w)$ and $f=f(w)$ are defined as above.
\end{corollary}

\begin{proof}
	Lemma~\ref{lem:bd_dim_rt_idl} and the comment preceding it imply that 
	$$
	\dim H^2(X(\pi(p)),\reals)\leq p^h-(p-f)^h.
	$$ 
	Normalizing, we obtain the desired identity.
\end{proof}

The homomorphisms $\pi(p)\colon K\twoheadrightarrow H(p)$, for $p\geq p_0=p_0(w)$ a sufficiently large prime, suffice now to prove the quantitative version of Theorem~\ref{thm:main_un_grps} given in the introduction.

Namely, one proves by induction on $n\geq 1$ that for $D=D(w)>0$ sufficiently large, 
$$
d_{\rk}(g,w(\U_n))\leq\varepsilon(n)\coloneqq (D\log(n)+1) n^{-1/h}
$$
for all $g\in\U_n$. Set $\varepsilon(0)\coloneqq 0$.

Indeed, this is true for $n<p_0^h$.
Now for $n\geq p_0^h$ we pick the largest prime $p$ such that $p^h\leq n$ and the largest integer $l\geq 1$ such that $lp^h\leq n$.
Then via the embedding $\U_{n-lp^h}\oplus\U_{p^h}^{\oplus l}\subseteq\U_n$, writing $n_1\coloneqq n-lp^h$ and $n_2\coloneqq lp^h$, we see that
$$
d_{\rk}(g,w(\U_n))\leq \frac{n_1}{n}\varepsilon(n_1)+\frac{n_2}{n}hf/p
$$
for all $g\in\U_n$ by the induction hypothesis and Corollary~\ref{cor:est_p_grps}. Since $p$ is the largest prime such that $p\leq n^{1/h}$, Bertrand's postulate implies that $p\geq n^{1/h}/2$. Moreover, by construction $n_1<n/2$, so that the above term can be bounded by
$$
\frac{1}{2}(D\log(n/2)+1)(n/2)^{-1/h}+2hfn^{-1/h}\leq (D\log(n)+1)n^{-1/h}
$$
again if $D$ is large enough.

\subsection{Further implications}\label{subsec:furth_impl}

It is easy to see that our method of proof implies that $w(\SU_n)$ has width at most two  in $\SU_n$ for $n$ large enough, which was first proven in \cite[Theorem~2.3]{huilarsenshalev2015waring} using Got{\^o}'s trick, Borel's theorem and the representation theory of $\SU_2$. The reason for this is the following basic fact about the linearized permutation representation of the Weyl group $\Sgrp_n$ of $\U_n$.

\begin{lemma}\label{lem:bd_wdth_perm_rep}
	Let $V=\reals^n$ be the permutation representation of $\Sgrp_n$. Let $V_0$ be the subrepresentation of $V$ of all vectors whose entries sum to zero. If $U_1,U_2\leq V_0$ are subspaces and
	$\dim(U_1)+\dim(U_2)\geq n-1$, then $U_1+U_2.\sigma=V_0$ for some $\sigma \in \Sgrp_n$.
\end{lemma}

\begin{proof} 
	This is a consequence of the fact that the exterior power $\Lambda^k(V_0)$ is irreducible for $k=0,\ldots,n-1$, see Proposition~3.12 of \cite{fultonharris1991representation}. Note that the determinant pairing
	$h\colon\Lambda^k(V_0) \times \Lambda^{n-1-k}(V_0)\to\reals$ given by $h(v_1\wedge\cdots\wedge v_k,v_{k+1}\wedge\cdots\wedge v_{n-1})=v_1\wedge\cdots\wedge v_{n-1}\in\Lambda^{n-1}(V_0)\cong\reals$, is non-degenerate. Using irreducibility, we can easily see that this implies the claim. Indeed, set $k\coloneqq\dim(U_1)$, choose a basis $u_{1},\ldots,u_{k}$ for $U_1$ and a linearly independent set $u'_{k+1},\ldots,u'_{n-1}$ in $U_2$. Now we have that $h(u_1 \wedge \ldots \wedge u_{k},(u'_{k+1} \wedge \ldots \wedge u'_{n-1}).\sigma) \neq 0$ if and only if $U_1 + U_2.\sigma = V_0$. By irreducibility, the set $\set{(u'_{k+1} \wedge \ldots \wedge u'_{n-1}).\sigma}[\sigma\in \Sgrp_n]$ spans $\Lambda^{n-1-k}(V_0)$. The fact that $h$ is non-degenerate implies the claim.
\end{proof}

We immediately obtain the following corollary.

\begin{corollary}
	Let $w_1,w_2\in\freegrp_2$ be non-trivial. Then $$w_1(\SU_n)w_2(\SU_n)=\SU_n$$ for $n$ sufficiently large.
\end{corollary}

\begin{proof} 
	Set $U_i\leq w_i(\reals\wr\Sgrp_n)\cap\reals^n\leq V\coloneqq\reals^n$ to be a vector subspace of the diagonal matrices $\reals^n$ which lies in the above $w_i$-image and has maximal dimension with respect to this property ($i=1,2$). In Lemma~\ref{lem:cohom_lem}, Corollary~\ref{cor:est_p_grps}, and the remarks thereafter, we have shown that $\dim(U_i)\geq\frac{n-1}{2}$ for $n$ large enough. Applying Lemma~\ref{lem:bd_wdth_perm_rep} to $U_1,U_2$ and $V$, and exponentiating, we see that for every diagonal matrix $g\in\SU_n$ there are $h_i\in w_i(\SU_n)$ ($i=1,2$) such that $g=h_1h_2^{\sigma}$, so that $w_1(\SU_n)w_2(\SU_n)=\SU_n$.
\end{proof}

\subsection{Concluding remarks}\label{subsec:con_rmk}

Lemma~\ref{lem:cohom_lem} and the above proof of Theorem~\ref{thm:main_un_grps} can be seen as a generalization of the methods used in \cite{elkasapythom2014goto} and clarify various aspects of it. Let us demonstrate this briefly. For a word $w\in\freegrp_2'\setminus\freegrp_2''$ set $\pi\colon K\twoheadrightarrow H\coloneqq\freegrp_2/\freegrp_2'=\ints^2=\gensubgrp{g,h}$ to be the natural homomorphism. Applying Lemma~\ref{lem:d_1*_nontriv} to $\pi$, we see that $d^2(\pi)$ is non-trivial, so again we get an edge $e\in\Gamma(\pi)$ such that $d^2(\pi)(e^\ast)\neq 0$, which corresponds to an element $z=z(g,h)\in  \ints[H]=\ints[\ints^2]=\ints[g^{\pm 1},h^{\pm 1}]$ in the integral group ring. By symmetry, it is no loss to assume that $e$ is rooted at $1_H$ and labeled by $h$.

The Laurent polynomial $p_w(X)$ defined in Section~3 of \cite{elkasapythom2014goto} is now precisely equal to $z^\ast(X,1)=z(X^{-1},1)$, where $z(g,h)=d^2(\pi)(e^\ast)=\pi(\partial w/\partial y)^\ast$ (see Remark~\ref{rmk:fox_der}). Here $z^\ast(X,1)$ is just the image of $z$ under the homomorphism $\ints[\ints^2]=\ints[g^{\pm 1},h^{\pm 1}]\twoheadrightarrow\ints[\ints]=\ints[X^{\pm1}]$ induced by $g^a h^b\mapsto X^{-a}$, as, e.g., for $w=[x^a, y^b]=x^{-a}y^{-b}x^a y^b$, $a,b>0$ we have $z(g,h)=(h+\cdots+h^b)(1-g^a)$ and $p_w(X)=-b(X^{-a}-1)$. 

Now we can find a suitable homomorphism $\varphi\colon\ints^2=\gensubgrp{g,h}\twoheadrightarrow\ints=\gensubgrp{X}$ such that the induced ring homomorphism $\ints[\ints^2]=\ints[g^{\pm1},h^{\pm1}]\twoheadrightarrow\ints[\ints]=\ints[X^{\pm 1}]$ maps $z$ to a non-zero element $\varphi(z)=p(X)$ (e.g., take as the kernel of the homomorphism $\varphi$ a saturated copy of $\ints$ in $\ints^2$ which does not hit any element in the support of $z$). For $n\in\ints_+$ we define the homomorphism $\pi(n)\colon K\twoheadrightarrow H(n)=\ints/(n)$ just by composing $\varphi\circ\pi$ with the natural projection $\ints\twoheadrightarrow\ints/(n)$. 
One now quickly derives the conclusions of Lemma~3.1, Corollary~3.2, and Proposition~3.8 of \cite{elkasapythom2014goto} from the following lemma.

\begin{lemma}
	Let $p(X)$ be as above. Write $z(n)$ for the image of $z$ in $\ints[H(n)]$.
	Define $W_n\coloneqq\set{\omega\in\complex}[p(\omega)=0\text{ and }\omega^n=1]$.
	The (right) ideal $z(n)\reals[H(n)]$ has codimension $\card{W_n}$, so in particular, if the least prime dividing $n$ is large enough, then it has codimension one and the word map $w$ on $\SU_n$ is surjective by Lemma~\ref{lem:cohom_lem} and Remark~\ref{rmk:surj_wrd_mp}.
\end{lemma}

\begin{proof}
	By the Chinese remainder theorem, we have the isomorphism 
	$$
	\reals[H(n)]=\reals[\ints/(n)]\cong\reals[X]/(X^n-1)\cong\bigoplus_{\substack{\chi\mid X^n-1\\\chi\text{ irreducible}}}{\reals[X]/(\chi)}\cong\reals^{e_n}\oplus\complex^{\ceil{n/2}-1},
	$$ 
	where $e_n=2$ if $n$ is even and $e_n=1$ if $n$ is odd. This holds, since the (monic) irreducible polynomials $\chi\mid X^n-1$ are either of the form $X\pm1$ (so that $\reals[X]/(\chi)\cong\reals$) or of the form $(X-\omega)(X-\overline{\omega})$ for $\omega\in\complex\setminus\reals$ an $n$th root of unity (so that $\reals[X]/(\chi)\cong\complex$). The last isomorphism in the above equation is given by $\overline{X}\mapsto(\omega_\chi)_\chi$, where $\omega_\chi$ is a root of $\chi$. Hence the ideal generated by $z(n)$ has as codimension precisely the number of $n$th roots $\omega$ for which $p(\omega)=0$, as claimed. The second claim follows from the fact that $p(X)\in\ints[X^{\pm 1}]$ and the minimal polynomial of a primitive $m$th root of unity, $m>1$ dividing $n$, over $\rats$ is the cyclotomic polynomial $\Phi_m(X)$ of degree $\varphi(m)\geq p-1$, where $p$ is the least prime divisor of $n$. Hence, if $p-1>\deg(p(X))$, we have $W_n=\set{1}$. This completes the proof.
\end{proof}

The above shows that the result from \cite{elkasapythom2014goto} is precisely the simplest application of Lemma~\ref{lem:cohom_lem}, namely when $G=H(n)$ is taken to be cyclic.
We can now also understand that Question~4.4 from \cite{elkasapythom2014goto} has a negative answer. Indeed, assume that for every choice of $\varphi\colon\ints^2=\gensubgrp{g,h}\twoheadrightarrow\ints=\gensubgrp{X}$ in the above construction, the image $\im(d^2(\pi(n)))\subseteq\reals[H(n)]\cong\reals[X]/(X^n-1)$ has codimension greater than one, i.e., $d^2(\varphi\circ\pi)(\varphi(g)\cdot 1)$ and $d^2(\varphi\circ\pi)(\varphi(h)\cdot 1)$ as Laurent polynomials in $\ints[X^{\pm1}]$ have a non-trivial $n$th root of unity as a common root. Now choose $\alpha\in\Aut(\freegrp_2)$ arbitrary. Replacing $w$ by $\alpha(w)$ will not improve this situation. To see this, set $x'\coloneqq\alpha(x)$, $y'\coloneqq\alpha(y)$
and obtain by the chain rule
\begin{align*}
d^2_{\alpha(w)}(\varphi\circ\pi)(\varphi(g)\cdot 1) &=\varphi\circ\pi\left(\frac{\partial w}{\partial x}(x',y')\frac{\partial x'}{\partial x}+\frac{\partial w}{\partial y}(x',y')\frac{\partial y'}{\partial x}\right)^\ast,\\
d^2_{\alpha(w)}(\varphi\circ\pi)(\varphi(h)\cdot 1) &=\varphi\circ\pi\left(\frac{\partial w}{\partial x}(x',y')\frac{\partial x'}{\partial y}+\frac{\partial w}{\partial y}(x',y')\frac{\partial y'}{\partial y}\right)^\ast,
\end{align*}
where $d^2_{\alpha(w)}(\varphi\circ\pi)$ denotes the map corresponding to $d^2(\varphi\circ\pi)$ but with $\alpha(w)$ in the role of $w$. But then, as $\alpha$ is an automorphism, we see that $\varphi\circ\pi(x')=X^a$ and $\varphi\circ\pi(y')=X^b$ with $a,b\in\ints$ coprime (as they must generate $\ints=\gensubgrp{X}$). Hence
$$
\varphi\circ\pi\left(\frac{\partial w}{\partial x}(x',y')\right)=\frac{\partial w}{\partial x}(X^a,X^b) \quad\text{resp.}\quad\varphi\circ\pi\left(\frac{\partial w}{\partial y}(x',y')\right)=\frac{\partial w}{\partial y}(X^a,X^b),
$$
which is equal to $d^2(\varphi'\circ\pi)(\varphi'(g)\cdot1)^\ast$ resp.\ $d^2(\varphi'\circ\pi)(\varphi'(h)\cdot1)^\ast$, where $\varphi'\colon\ints^2=\gensubgrp{g,h}\twoheadrightarrow\ints=\gensubgrp{X}$ is given by $g\mapsto X^a$, $h\mapsto X^b$ (see Remark~\ref{rmk:fox_der}). But these two expressions seen as Laurent polynomials in $\ints[X^{\pm1}]$ by our assumption have a non-trivial $n$th root of unity $\omega$ as a common root. But then, by the above equations, $d^2_{\alpha(w)}(\varphi\circ\pi)(\varphi(g)\cdot 1)$ and $d^2_{\alpha(w)}(\varphi\circ\pi)(\varphi(h)\cdot 1)$ also must have $\omega$ as a root, so that the image $\im(d^2_{\alpha(w)}(\pi(n)))\subseteq\reals[H(n)]\cong\reals[X]/(X^n-1)$ has codimension greater than one.

In retrospect, as has been pointed out to us by Jack Button, the study in \cite{elkasapythom2014goto} would have been much clearer, when the connection to Fox calculus and the even more classical subject of Alexander polynomials would have been observed from the start.


\vspace{0.1cm}

Let us end this section by drawing some further connections to related facts.
In case that $K$ is residually finite, one could also prove Theorem~\ref{thm:main_un_grps} using L{\"u}ck's approximation theorem together with the fact that the second $\ell^2$-Betti number of a one-relator group is zero by a well-known result of Dicks and Linnell~\cite{dickslinnell2007l2} -- or the validity of the $\ell^2$-zero divisor conjecture for torsionfree nilpotent groups applied to $H$ (see \cite{lueck2002L2} for more background). However, our argument is much more explicit and does even give an effective estimate.

\section{Finite groups of Lie type}\label{sec:fin_grps_lt}

In this section, we prove Theorem~\ref{thm:main_thm_fin_grp_lt} using aspects presented in Sections~\ref{sec:sym_grps} and \ref{sec:un_grps} -- for convenience of the reader we decided to present the proof first in the case of unitary groups $\U_n$, where the methods come into play in the most natural way. However, we will now use the same cohomological method as in Lemma~\ref{lem:cohom_lem} together with Lemmas~\ref{lem:d_1*_nontriv} and \ref{lem:bd_dim_rt_idl} of Section~\ref{sec:un_grps}, but instead of using the additive group of $\reals$ as our coefficient group, we now will use groups of type $(\finfield_q[X]/(\chi))^\times$ for $\chi\in\finfield_q[X]$ some polynomial. Indeed, we will need the following modified version of Lemma~\ref{lem:bd_dim_rt_idl}, which is an easy consequence of it.

\begin{corollary}\label{cor:bd_dim_rt_idl_ints}
	In the setting of Lemma~\ref{lem:bd_dim_rt_idl}, using coefficients in $\ints$ instead of the field $k$, there are a non-zero $c\in\nats$ and $f\in\nats$ such that for all primes $p$, there exists a subset $C\subseteq H(p)$ of at least $(p-f)^h$ coordinates so that the projection of the right ideal $y(p)\ints[H(p)]$ onto $\ints[C]$ contains the $\ints$-module $(c\ints)[C]$. Here $\ints[C]$ are all $\ints$-linear combinations of elements from $C$ inside the ring $\ints[H(p)]$. Moreover, $(c\ints)[C]\coloneqq c(\ints[C])$.
\end{corollary}

\begin{proof}
	Applying Lemma~\ref{lem:bd_dim_rt_idl} to $k=\rats$, we get a set $C\subseteq H(p)$ of coordinates of size $\card{C}\geq(p-f)^h$ such that $y(p)\rats[H(p)]$ projects surjectively on these. Hence we generate the unit vectors in $\rats[H(p)]/\rats[H(p)\setminus C]$. So multiplying by the least common multiple $c$ of the denominators of the involved scalars, we obtain that the projection of $y(p)\ints[H(p)]$ onto the coordinates $C$ still contains the module $(c\ints)[C]$. 
\end{proof}

Subsequently, we fix the symbol $c$ to be the constant from Corollary~\ref{cor:bd_dim_rt_idl_ints}.
For a polynomial $\chi\in k[X]$ for a field $k$, write $F(\chi)$ for the \emph{Frobenius block} associated to $\chi$, that is the matrix of multiplication by $X$ in $k[X]/(\chi)$ with respect to the standard monomial basis. Similarly, for $\lambda\in\overline{k}$ write $J_e(\lambda)$ for the \emph{Jordan block} of size $e$ with respect to $\lambda$, that is the matrix of multiplication by $\lambda+X$ in $k[\lambda,X]/(X^e)$ with respect to the basis $1,\overline{X},\ldots,\overline{X}^{e-1}$. Call a polynomial $\chi\in k[X]$ \emph{primary} if it is a power of an irreducible polynomial, i.e., if the ideal it generates is primary. Recall that for an element $g\in\End(V)$, $V$ is irreducible resp.\ indecomposable resp.\ cyclic as a $k[X]$-module, where $X$ acts as $g$, if and only if $g\cong F(i)$ resp.\ $g\cong F(\chi)$ resp.\ $\bigoplus_{i=1}^l{F(\chi_i)}\cong F(\chi_1\cdots\chi_l)$ for an irreducible polynomial $i\in k[X]$ resp.\ a primary polynomial $\chi\in k[X]$ resp.\ pairwise coprime primary polynomials $\chi_1,\ldots,\chi_l\in k[X]$.

\subsection{The linear case.}\label{subsec:lin_cs}

We start by proving Theorem~\ref{thm:main_thm_fin_grp_lt} in the case when $G=\GL_n(q)$. We consider here the more general case that $G=\GL_n(k)$ for an arbitrary field $k$. So let $V=k^n$ be the natural module of $G$. We use the same approach as in Subsection~\ref{subsec:prf_mn_thm_sym_grps}, first approximating isotypic elements $g\in\GL(V)$ by word values, i.e., we first assume that $V$ is the direct sum of isomorphic cyclic $k[X]$-submodules so that $g\cong F(\chi)^{\oplus c_\chi}$ for some polynomial $\chi\in k[X]$ of degree $m$, and then deducing the general case by using the Frobenius normal form and Jensen's inequality.

\emph{The isotypic case.} So let $\chi$, $c_\chi$, and $m$ be as previously mentioned. We want to approximate $F(\chi)$-isotypic elements by word values with these parameters, so that $n=c_\chi m$. As in Subsection~\ref{subsec:prf_mn_thm_sym_grps}, we distinguish two cases, one in which $m$ is small and one in which it is large (compared to $c_\chi$).

\emph{Estimate for small $m$.} In view of Corollary~\ref{cor:bd_dim_rt_idl_ints}, we need the following auxiliary fact.

\begin{lemma}\label{lem:pow_jor_blks}
	It holds that $F(\chi(X^c))^c\cong F(\chi)^{\oplus c}$.
\end{lemma}

\begin{proof}
	The block $F(\chi(X^c))$ is the matrix of multiplication by $X$ in the ring $k[X]/(\chi(X^c))$, so that $F(\chi(X^c))^c$ is the multiplication by $X^c$ in $k[X]/(\chi(X^c))$. But $k[X]/(\chi(X^c))=\bigoplus_{i=0}^{c-1}\overline{X}^i\gensubalg{\overline{X}^c}_k$ holds for dimension reasons, so that the claim follows.
\end{proof}

Now we use the same idea as in Lemma~\ref{lem:cohom_lem} with appropriate coefficient group. Consider the ring $R\coloneqq k[X]/(\chi(X^c))$ and write $c_\chi=rc+s$ for $r\in\nats$ and $0\leq s<c$.
Corollary~\ref{cor:bd_dim_rt_idl_ints} and Lemma~\ref{lem:pow_jor_blks} give us that in $R^\times\wr\Sym(r)\subseteq\GL_{cmr}(k)$ we have that $w(R^\times\wr\Sym(r))$ approximates the block diagonal matrix $(F(\chi(X^c))^{\oplus r})^c\cong F(\chi)^{\oplus cr}$ up to an error of $d(1/r)$.
Hence, since the function $d$ is concave, we obtain
$$
d_{\rm rk}(g,w(\GL_n(k)))\leq \frac{cr}{c_\chi}d(1/r)+s/c_\chi<d(c/c_\chi)+c/c_\chi.
$$

\emph{Estimate for large $m$.} On the other hand, the matrices $F(\chi)$ and $F(X^m-1)$ differ only in the last row, so by rank one. The last matrix is the permutation matrix of an $m$-cycle, which we can approximate by word values by the result for symmetric groups (Theorem~\ref{thm:main_sym_grps}). Hence $d_{\rm rk}(F(\chi),F(X^m-1))\leq 1/m$, implying that
$$
d_{\rk}(g,w(\GL_n(k)))<d(1/m)+1/m.
$$

\emph{Global estimate for isotypic elements.} Now we combine both estimates, as in the proof for symmetric groups.
Using the first estimate if $m\leq\sqrt{n/c}$ and the second in the opposite case, we obtain 
$$
d_{\rk}(g,w(\GL_n(q)))<d\left(\sqrt{c/n}\right)+\sqrt{c/n}
$$
as wished. Subsequently, in analogy to the proof of Theorem~\ref{thm:main_sym_grps} in Section~\ref{sec:sym_grps}, write $d_{\rm it}$ for the function of $1/n$ on the right.

\emph{The general case.} Using the Frobenius normal form we can write $g\cong\bigoplus_{m\geq 1}{F(\chi_m)^{\oplus c_m}}$, $\chi_m$ being the invariant factor of degree $m$ and $c_m\coloneqq c_{\chi_m}$.

Now we can finish the proof. Writing $n_m\coloneqq c_m m$, we get that

\begin{align*}
d_{\rk}(g,w(\GL_n(k)))&\leq\sum_{m\geq 1}{\frac{n_m}{n}d_{\rm it}(1/n_m)}\\
&\leq d_{\rm it}\left(\sum_{m\geq 1}{1/n}\right)\leq d_{\rm it}\left(\sqrt{2/n}\right).
\end{align*}
as in the end of Subsection~\ref{subsec:prf_mn_thm_sym_grps}, as desired.

\begin{remark}
	Similarly to the symmetric case, one verifies that such a bound is also attained for power words $w=x^p$ when $\rchar(k)=p$.
\end{remark}

\subsection{The case of nearly simple groups of Lie type stabilizing a form.}

We proceed by proving Theorem~\ref{thm:main_thm_fin_grp_lt} for nearly simple groups of Lie type of unbounded rank which stabilize a form, i.e., subsequently $G$ is of the form $\Sp_{2n'}(q)$, $\GO_{2n'+1}(q)$, $\GO_{2n'}^{\pm}(q)$ or $\GU_n(q)$ ($n\geq 2$, $n'\geq 1$).

Subsequently, let $k=\finfield_q$ if $G$ is one of $\Sp_{2n'}(q)$, $\GO_{2n'+1}(q)$, $\GO_{2n'}^{\pm}(q)$, and $k=\finfield_{q^2}$ in case $G=\GU_n(q)$ for some $n\geq 2$ or $n'\geq 1$. Write $(V,f)$ for the natural module of $G$, i.e., $V=k^n$ and $f$ is a non-singular alternating bilinear form ($G=\Sp_{2n'}(q)$), a non-singular symmetric bilinear form ($G=\GO_{2n'+1}(q)$ or $G=\GO_{2n'}^\pm(q)$, $q$ odd), or a conjugate-symmetric sesquilinear form ($G=\GU_n(q)$). In case that $p\coloneqq\rchar(k)=2$, and $G=\GO_n(q)$, $f$ comes from a non-singular quadratic form $Q$, so that $f$ is alternating. Note that we can neglect the case that $G=\GO_{2n'+1}(q)$ for $q$ even ($n'\geq 1$), where $f$ is singular, since then $G\cong\Sp_{2n'}(q)$ via $g\mapsto\overline{g}\in\GL(V/\rad(f))$ (see, e.g., \cite{wilson2009finite}).

In the unitary case, $f$ is semilinear in the second entry with respect to the $q$-Frobenius endomorphism $x\mapsto x^q$. The sign $\varepsilon\in\set{\pm1}$ is defined to be $+1$ if $f$ is symmetric bilinear or conjugate-symmetric sesquilinear, and to be $-1$ if $f$ is alternating. Similarly, the automorphism $\sigma$ of $k$ is defined to be the $q$-Frobenius endomorphism in the unitary case, and the identity in the bilinear case.

For a fixed $g\in G$, which we want to approximate by word values, subsequently consider $(V,f)$ as a $k[X]$-module, where $X$ acts as $g$. A non-singular submodule of $V$ is said to be \emph{orthogonally indecomposable} if it is not an orthogonal direct sum of non-trivial proper submodules (with respect to the form $f$).

In analogy to the linear case, $V$ is the orthogonal direct sum of such submodules. Hence, following the same strategy as in Subsection~\ref{subsec:lin_cs}, we first consider the case when $V$ is itself orthogonally indecomposable. We recall the classification of such modules $V$ (all statements are well known and are, e.g., used in \cite[Section 6]{liebeckshalev2001diameters}; see also \cite{wall1963conjugacy} and \cite{gonshawliebeckobrien2017unipotent} for the unipotent case; for a unified treatment we refer to Subsection~3.4.2~\S1 of \cite{schneider2019phd}). For a monic polynomial $\chi=a_0+a_1X+\cdots+a_{m-1}X^{m-1}+X^m\in k[X]$, write $\chi^\ast\coloneqq a_0^{-\sigma}X^m\chi^\sigma(X^{-1})$ for its \emph{dual polynomial}, where $\chi^\sigma$ is the polynomial $\chi$ with coefficients twisted by $\sigma$, and say $\chi$ is self-dual if $\chi=\chi^\ast$. For a module $U$, write $U^\ast$ for the module of $\sigma$-semilinear functionals on $U$.\vspace{0.2cm}

\subsubsection{Structure of orthogonally indecomposable modules}\label{subsubsec:str_indec}

We distinguish into three cases (for a detailed discussion of these, we refer to~\cite{schneider2019phd}).

\emph{Case 1: The non-self-dual case.} $V\cong U\oplus U^\ast$, where $\rest{g}_U\cong F(i^e)$ and $\rest{g}_{U^\ast}\cong F(i^{\ast e})$ for $i\in k[X]$ non-self-dual irreducible and $e\geq 1$ (so that $g\cong F((ii^\ast)^e)$ on $V$). The form $f$ is given by the dual pairing $f(u,u^\ast)=\varepsilon f(u^\ast,u)^\sigma=(u^\ast(u))^\sigma$ for $u\in U, u^\ast\in U^\ast$ and $\rest{f}_U=\rest{f}_{U^\ast}=0$.
So $f$ is uniquely determined by $g$ up to equivalence.

\emph{Case 2: The self-dual case when $i\neq X\pm 1$.} In this case, $g\cong F(\chi)=F(i^e)$ for self-dual polynomials $\chi$ resp.\ $i\neq X\pm1$ which are primary resp.\ irreducible. Again $f$ is uniquely determined by $g$ here. Set $C\coloneqq k[X]/(\chi)$ and let $\alpha\in\Aut(C)$ be the map inducing $\sigma$ on $k$ and sending $\nu\coloneqq\overline{X}$ to $\nu^{-1}$ (which is an automorphism, since $\chi$ is self-dual). Then identifying $V$ with $C$, $f$ is given by $(u,v)\mapsto\ell(uv^\alpha)$, where $\ell\colon C\to k$ is an appropriate linear form such that $\ell(u)^\sigma=\varepsilon\ell(u^\alpha)$ for $u\in C$.

Let $\lambda\in\overline{k}$ be a root of $i$. First assume that $f$ is bilinear. Then, since $i\neq X\pm 1$, we have that $\lambda\neq\lambda^{-1}$, so that $d=\deg(i)$ is even, i.e., $i=i'(X+X^{-1})X^{d/2}$ for an irreducible polynomial $i'\in k[X]$. Next assume that $f$ is $\sigma$-sesquilinear and $\lambda^2\neq 1$. Consider the field extension $k[\lambda]\supset k\supset k_\sigma$. The automorphism $\alpha$ of $C$ descends to an automorphism $\overline{\alpha}$ of $k[\lambda]$ which maps $\lambda\mapsto\lambda^{-1}$ and restricts to $\sigma\colon x\mapsto x^q$ on $k=\finfield_{q^2}$. As $k[\lambda]=\finfield_{q^{2d}}$ is a finite field, $\overline{\alpha}$ is the unique involution $x\mapsto x^{q^d}$. But $\overline{\alpha}$ must induce $\sigma$ on $k$, so that $d$ must be odd. Now one observes that $k[\lambda]\supset k_\sigma$ must be of even degree, as its Galois group contains the involution $\overline{\alpha}$, so that the minimal polynomial of $\lambda$ over $k_\sigma=\finfield_q$ is (up to constant factor) $ii^\sigma$ of degree $2d$. Again define the $k_\sigma$-irreducible polynomial $i'\in k_\sigma[X]$ of degree $d$ by $ii^\sigma=i'(X+X^{-1})X^d$.
Finally, in both cases one verifies easily that, setting $\chi'\coloneqq i'^e\in k_\sigma[X]$, $i'$ resp.\ $\chi'$ is the characteristic polynomial of $\lambda+\lambda^{-1}$ resp.\ $\nu+\nu^{-1}$ in $k[\lambda]$ resp.\ $C$ over $k_\sigma$.

\emph{Case 3: The case that $i=X\pm 1$.} This case is extensively discussed by \cite[Proposition~2.2, 2.3, 2.4, and Theorem~3.1]{gonshawliebeckobrien2017unipotent}. Note the following fact, which we will use later: Assume $g$ is isotypic with many orthogonally indecomposable summands of this type. If $p\neq 2$, then again $g$ determines $f$ up to equivalence on all but at most one indecomposable summand, as is shown in \cite[Proposition~2.2, 2.3, 2.4]{gonshawliebeckobrien2017unipotent}. If $p=2$, then choosing coordinates such that $f$ is in the normal form of \cite[Theorem~3.1]{gonshawliebeckobrien2017unipotent}, we see that $f$ restricted to all but constantly many Jordan blocks is of the form $W(e)$.

\begin{remark}
	In each of the cases, when $\rchar(k)=p=2$, there exists a suitable quadratic form inducing $f$. However, we do not need its explicit form.
\end{remark}

\emph{The Frobenius normal form for elements $g\in G$.}
We wish to apply the same method as in Subsection~\ref{subsec:lin_cs}, for which we need an analog of the Frobenius normal form for elements $g\in G$.

Write $g=h\perp u=h\perp u_1\perp u_{-1}$, where $u_1,-u_{-1}$ are unipotent and $h$ has only eigenvalues different from $\pm 1$. This is possible by considering the Cases~1,~2, and~3 of indecomposables. 

We obtain a normal form for $h$ in the same way as the Frobenius normal form is obtained from the primary canonical form: In the first summand we collect all orthogonally indecomposable summands from Case~1 resp.\ Case~2 of the form $F((ii^\ast)^e)=F(i^e)\oplus F(i^{\ast e})$ with $i\in k[X]$ irreducible and non-self-dual resp.\ $F(i^e)$ with $i\in k[X]$ irreducible and self-dual (and $i\neq X\pm1$), and $e$ as large as possible. Then we split off this summand and proceed in the same way with the perpendicular complement.

For $u_1$ and $u_{-1}$ we use the normal form provided by \cite[Proposition~2.2, 2.3, 2.4, and Theorem~3.1]{gonshawliebeckobrien2017unipotent}.

We still need the following fact, which follows from the analysis of Cases~1,~2, and~3:
\begin{fact}\label{fct:ex_frm} 
Whenever $\chi\in k[X]$ is self-dual and is not divisible by $X\pm1$ in the bilinear case, then there exists a non-singular form $f$ (coming from a quadratic form $Q$ when $p=2$) which is preserved by $F(\chi)$ ($f$ is even unique up to linear equivalence). On the other hand, for all $e$, we have that there is a form $f$ (together with $Q$ when $p=2$) preserving $F((X\pm1)^e)^{\oplus 2}$.
\end{fact}

\subsubsection{Proof of Theorem~\ref{thm:main_thm_fin_grp_lt} for the remaining groups}
We decompose $g\cong h\perp u_1\perp u_{-1}$ as described above. Hence, using Jensen's inequality, we only need to consider two cases: (a) $g=h$ and (b) $g=\pm u$, where $u$ is unipotent. Now we can apply the previous considerations to elements that are $F(\chi)$-isotypic for $\chi\in k[X]$ of degree $m$ which is not divisible by $X\pm 1$ in Case (a), and elements that are $F((X\pm1)^e)^{\oplus d}$-isotypic, where $d=1$ or $2$ and $m=de$ in Case (b). It is enough to consider these two isotypic cases. Again we derive an estimate for $m$ small and $m$ large. 

\emph{Estimate for small $m$.} In Case~(a), we have that $g\cong F(\chi)^{c_\chi}$ for $\chi$ self-dual and there is up to equivalence only one form $f$ preserved by $F(\chi)$ (which follows from the first part of Fact~\ref{fct:ex_frm} above). 
We can approximate the linear map $g$ by elements from $w(\gensubgrp{\overline{X}}\wr\Sym(r))$, where $\gensubgrp{\overline{X}}\subseteq R^\times=(k[X]/(\chi(X^c)))^\times$ and $c_\chi=rc+s$ for $0\leq s<r$, as in the estimate for small $m$ in Subsection~\ref{subsec:lin_cs}. But $\overline{X}\in R^\times$ preserves a non-singular form $f$ as $\chi(X^c)$ is again self-dual (which again follows from Fact~\ref{fct:ex_frm} above), so that also the group $\gensubgrp{\overline{X}}$ preserves such a form and we are done by uniqueness.

In Case~(b), we have $g\cong F((X\pm1)^e)^{\oplus dc_{d,e}}$. Since $m=de$ is small, $c_{d,e}$ is large and we can certainly assume it to be even. Observe that $F((X\pm1)^e)^{\oplus 2}$ always supports a non-singular form, so that $F((X^c\pm1)^e)^{\oplus 2}$ will do as well (by Fact~\ref{fct:ex_frm}). Hence we can use the same trick as in Case~(a) and use Fact~\ref{fct:ex_frm} in the unitary case and \cite[Propositions~2.3,~2.4, and Theorem~3.1]{gonshawliebeckobrien2017unipotent} in the bilinear case, which says that the form $f$ is essentially determined by $g$ up to a constant number of summands $F((X\pm1)^e)^{\oplus d}$ (namely, with the notation used there, most of its blocks will be $U(e)$ in the unitary case and $V_1(e)$ or $W(e)$ in the bilinear case).

\emph{Estimate for large $m$.} In this case, we assume that $g\cong F(\chi)$ in Case~(a) (so $c_\chi=1$) and $g\cong F((X\pm1)^e)^{\oplus d}$ for $d=1$ or $2$. Here we want to apply the following simple fact.

\begin{lemma}\label{lem:alm_inv_tot_sing_subsp_approx}
	Let $C>0$ be a fixed constant. Assume that $V=X\oplus Y\oplus Z$, where $X$ and $Y$ are totally isotropic, $n-2\dim(X),n-2\dim(Y)\leq C$, i.e., $X$ and $Y$ are close to a Witt subspace, and $\codim_X(X\cap X.g),\codim_Y(Y\cap Y.g)\leq C$, i.e., $X$ and $Y$ are almost $g$-invariant. Then $g$ can be approximated by word values.
\end{lemma}

\begin{proof}
	Note that $\dim(Y^\perp)=n-\dim(Y)\leq\frac{n+C}{2}$, so that $\dim(X\cap Y^\perp)\leq C$. Hence we can find $X'\leq X$ of dimension at least $\dim(X)-C\geq\frac{n-3C}{2}$ such that $\rest{f}_{X'\times Y}$ is separating in $X'$, i.e., the ma¸p $x\mapsto f(x,\bullet)$ is injective on $X'$. Hence, choosing $Y'\leq Y$ which induces all $\sigma$-semilinear functionals $X'^\ast$, we can assume by passing from $X$ to $X'$ and $Y$ to $Y'$ that $\rest{f}_{X\times Y}$ is non-degenerate, so in particular $\dim(X)=\dim(Y)$.
	
	Now let $g'$ be an extension of $\rest{g}_{X\cap X.g^{-1}}\colon X\cap X.g^{-1}\to X\cap X.g$ to an invertible linear map $X\to X$. By Subsection~\ref{subsec:lin_cs}, we find $h=w(x,y)\in w(\GL(X))$ such that $d_{\rk,X}(g',h)\leq d(1/\dim(X))\leq d(\frac{2}{n-C})$.
	
	We extend $h$ to all of $V$ as follows: Extend $x,y\in\GL(X)$ to $Y$ by taking their dual on $Y$, so that they fix $X$ and $Y$ setwise, and then extend them to $V$ with Witt's lemma. Then set $h\coloneqq w(x,y)$ on all of $V$. Write $Y.(g-h)=(Y.(g-h)\cap Y)\oplus W$. Then, since $h$ fixes $Y$, $W$ is injectively mapped by the natural map $Y.(g-h)\to (Y+Y.g)/Y$, but the last quotient, by assumption, had dimension at most $C$, so that $\dim(W)\leq C$. Now $f(x.(g^{-1}-h^{-1}),y)=f(x,y.(g-h))=0$ for $x\in X$, $y\in Y$, when $x\in\ker(g^{-1}-h^{-1})=(\ker(g-h)).h$ as $g^{-1}-h^{-1}=h^{-1}(h-g)g^{-1}$.
	But the vector space of all such $x\in X$ has dimension at least $\dim(X)(1-d(1/\dim(X))-C/\dim(X))$, which follows from the above estimate on $d_{\rk,X}(g',h)$ and the fact that $g$ and $g'$ agree on $X\cap X.g^{-1}$.
	
	Hence the dimension of $Y.(g-h)\cap Y$ is at most $\dim(X)(d(1/\dim(X))+C/\dim(X)))$, so that, using $\dim(W)\leq C$, the dimension of $Y.(g-h)$ is at most $\dim(X)(d(1/\dim(X))+2C/\dim(X))$. Hence the rank of $g-h$ is small on $X$ and $Y$, so is small on $V$. This ends the proof.
\end{proof}

Now $V$ is the direct sum of orthogonally indecomposable modules, each type of which occurs at most once. Subsequently, we construct subspaces $X$, $Y$, and $Z$ with the property required by Lemma~\ref{lem:alm_inv_tot_sing_subsp_approx}. Write $\mathcal S_j$ ($j=1,2,3$) for the orthogonally indecomposable summands of $V$ described in Case~$j$ from above.

For each orthogonally indecomposable summand $S=U\oplus U^\ast\in\mathcal S_1$ of $V$ as in Case~1, set $X_S\coloneqq U$, $Y_S\coloneqq U^\ast$, and $Z_S\coloneqq 0$. Then define $X_1\coloneqq\bigoplus_{S\in\mathcal S_1}{X_S}$, $Y_1\coloneqq\bigoplus_{S\in\mathcal S_1}{Y_S}$, and $Z_1\coloneqq 0$. Define $\chi_1$ by the fact that all the summands from Case~1 grouped together act as $F(\chi_1)$.

In Case~2, for each $S=U\in\mathcal S_2$ we have that $g$ acts as $F(\chi_S)=F(i_S^{e_S})$ on $S$, where $i_S$ is irreducible of degree $d_S$ and $\chi_S$ is of degree $m_S=d_Se_S$. Set $\chi_2\coloneqq\prod_{S\in\mathcal S_2}{\chi_S}\in k[X]$ and set $m_2\coloneqq\deg(\chi_2)$.
The form $f$ on $\bigoplus\mathcal S_2$ is given by $(u,v)\mapsto \ell(uv^\alpha)=\sum_{S\in\mathcal S_2}{\ell_S(u_Sv_S^{\alpha_S})}$, where $u=(u_S)_{S\in\mathcal S_2}$, $v=(v_S)_{S\in\mathcal S_2}$, and $\ell_S,\alpha_S$ ($S\in\mathcal S_2$) are as described in Case~2 above. Set $\nu\coloneqq(\nu_S)_{S\in\mathcal S_2}\in C\coloneqq\prod_{S\in\mathcal S_2}{C_S}$, $\alpha\coloneqq(\alpha_S)_{S\in\mathcal S_2}$ (cf.\ Case~2), and recall that $g$ acts on $V\cong C$ as multiplication by $\nu$.

That the vectors $v,\ldots,v.g^{l-1}$, for $v\in\bigoplus\mathcal S_2$, span a totally singular subspace hence means that $\ell(\N_\alpha(v)\nu^j)=0$ for $j=0,\ldots,l-1$, where we write $\N_\alpha(v)=vv^\alpha$. Write $u=\N_\alpha(v)$. We demonstrate how to find such a vector $v$ only in the orthogonal case when $p\neq 2$ (for simplicity). The other cases are similar. So assume we are in the orthogonal case. Define $i_S',\chi_S'\in k[X]$ for $S\in\mathcal S_2$ as in Case~2 above, and $\chi_2'$ by $\chi_2'\coloneqq\prod_{S\in\mathcal S_2}{\chi_S'}\in k[X]$, so that $\deg(\chi_2')=m_2/2$. Set $l\coloneqq\deg(\chi_2')-1$. Note that $C_\alpha\coloneqq\set{c\in C}[c^\alpha=c]$ is a $k$-subalgebra of $C$ of $k$-dimension $\deg(\chi_2')$ and $\ell$ descends to a $k$-linear functional $C_\alpha\to k$. The minimal polynomial of $\nu+\nu^{-1}\in C_\alpha$ over $k$ is $\chi_2'$, so that $C_\alpha=k[\nu+\nu^{-1}]\cong k[X]/(\chi_2')$. This implies that $(\nu+\nu^{-1})^j$ and hence $\nu^j+\nu^{-j}$ ($j=0,\ldots,l-1$) span an $l$-dimensional $k$-subspace of $C_\alpha$ and are hence linearly independent.

Now note that $\ell(\nu^j u)=0$ is equivalent to $\ell(\nu^{-j}u)=0$ when $u\in C_\alpha$, as $\ell$ has the property $\ell(x)=\ell(x^\alpha)$ and $\nu^\alpha=\nu^{-1}$.
But for $j\leq l-1$, $\nu^j+\nu^{-j}\neq 0$ from the previous observation, so that the two preceding equations are equivalent to $\ell((\nu^j+\nu^{-j})u)=0$. Now, from the construction of $\ell$, one sees that $f$ restricts to $C_\alpha$ as a non-singular form (see \cite{schneider2019phd}; here we use $p\neq 2$). Hence $R\coloneqq\gensubsp{(\nu+\nu^{-1})^j}[j=0,\ldots,l-1]^\perp\cap C_\alpha$ is one-dimensional. Let $0\neq u\in R$ be a generator of this subspace. We show that $u$ is a unit in $C_\alpha\subset C$. Assume the contrary, namely that $u_S\in(i_S)$ for some $S\in\mathcal S_2$. Then the $k$-linear functional $C_\alpha\to k$ which is zero on all $C_{S',\alpha_{S'}}$ ($S\neq S'\in\mathcal S_2$) and which equals $x\mapsto\ell(r i_S'^{e_S-1}(\nu_S+\nu_S^{-1})x)$ for $r\in C_{S,\alpha_S}^\times$ arbitrary must be a linear combination of the functionals $C_\alpha\to k$; $x\mapsto\ell((\nu+\nu^{-1})^j x)$ ($j=0,\ldots,l-1$). This means that there is a polynomial $s\in k[X]$ of degree $l-1=\deg(\chi_2')-2$ such that $s(\nu+\nu^{-1})$ is zero on $C_{S',\alpha_{S'}}$ for $S'\neq S$ and lies in $(i_S'^{e_S-1}(\nu_S+\nu^{-1}_S))=(i_S^{e_S-1})\cap C_{S,\alpha_S}$. This means that $\chi'_{S'}\mid s$ for $S'\neq S$ and $i_S'^{e_S-1}\mid s$. Hence, since the polynomials $i_T'\in k[X]$ ($T\in\mathcal S_2$) are irreducible and pairwise coprime, to achieve an arbitrary $r$, we hence need that $s=(\chi'_2/i'_S)s_0$, where $s_0\in k[X]$ is arbitrary of degree less than $d_S/2$. Hence we would need in the worst case that $\deg(s)=\deg(\chi_2')-1$, which is a contradiction. So we have that $u$ is a unit in $C$ and every unit is in the image of the norm $\N_\alpha\colon C\to C_\alpha$, since $k$ is a finite field (see~\cite{schneider2019phd} for the details), so that we find an appropriate vector $v\in C$ such that $\N_\alpha(v)=u$. Also, since $v$ is a unit in $C$, the space $X_2\coloneqq\gensubsp{v.g^j}[j=0,\ldots,\deg(\chi_2')-2]$ is actually of dimension $l$. Hence, setting $Y_2\coloneqq X_2.g^l$ and choosing $Z_3$ appropriately, we are done in this case.

Now in Case (a), we have that $g$ acts as $F(\chi)=F(\chi_1\chi_2)$. Setting $X\coloneqq X_1\oplus X_2$, $Y\coloneqq Y_1\oplus Y_2$, and $Z\coloneqq Z_1\oplus Z_2$ (and cutting of a further dimension if necessary when $p=2$ by restricting to $Q=0$), we can apply Lemma~\ref{lem:alm_inv_tot_sing_subsp_approx} to $F(\chi)$.

In Case (b), when $S\in\mathcal S_3$ one can easily extract almost invariant isotropic subspaces $X_S,Y_S$ and a space $Z_S$ from the explicit representations given in \cite[Propositions~2.2, 2.3,~2.4, and Theorem~3.1]{gonshawliebeckobrien2017unipotent} and sum them up as in Case (a).

\emph{Final proof.} The final proof is now identical with the one given in the last two paragraphs of Subsection~\ref{subsec:lin_cs}.

\section*{Acknowledgments}

We want to thank Vadim Alekseev and Sebastian Manecke for interesting discussions. The results presented in this paper are part of the PhD project \cite{schneider2019phd} of the first author. This research was supported in part by the ERC Consolidator Grant No.\ 681207.


\begin{bibdiv}
\begin{biblist}
\bib{avnigelanderkassabovshalev2013word}{article}{
	title={Word values in $p$-adic and adelic groups},
	author={Avni, Nir},
	author={Gelander, Tsachik},
	author={Kassabov, Martin},
	author={Shalev, Aner},
	journal={Bulletin of the London Mathematical Society},
	volume={45},
	number={6},
	pages={1323--1330},
	year={2013},
	publisher={Wiley Online Library}
}
\bib{bandmangariongrunewald2012surjectivity}{article}{
   author={Bandman, Tatiana},
   author={Garion, Shelly},
   author={Grunewald, Fritz},
   title={On the surjectivity of Engel words on ${\rm PSL}(2,q)$},
   journal={Groups, Geometry, and Dynamics},
   volume={6},
   date={2012},
   number={3},
   pages={409--439}
}
\bib{borel1983free}{article}{
	title={On free subgroups of semisimple groups},
	author={Borel, Armand},
	journal={L'Enseignement Math{\'e}matique},
	volume={29},
	number={1-2},
	pages={151--164},
	year={1983}
}
\bib{chowla1934}{article}{
	author={Chowla, Sarvadaman},
	title={The least prime congruent to one modulo $n$},
	journal={Journal of the Indian Mathematical Society},
	volume={1},
	number={2},
	year={1934},
	pages={1--3}
}
\bib{dickslinnell2007l2}{article}{
	title={$L^2$-Betti numbers of one-relator groups},
	author={Dicks, Warren},
	author={Linnell, Peter},
	journal={Mathematische Annalen},
	volume={337},
	number={4},
	pages={855--874},
	year={2007},
	publisher={Springer Science \& Business Media}
}
\bib{doughertymycielski1999representations}{article}{
	title={Representations of infinite permutations by words (II)},
	author={Dougherty, Randall},
	author={Mycielski, Jan},
	journal={Proceedings of the American Mathematical Society},
	volume={127},
	number={8},
	pages={2233--2243},
	year={1999}
}
\bib{dowerkthom2019bounded}{article}{
	title={Bounded normal generation and invariant automatic continuity},
	author={Dowerk, Philip},
	author={Thom, Andreas},
	journal={Advances in Mathematics},
	volume={346},
	pages={124--169},
	year={2019},
	publisher={Elsevier}
}
\bib{elkasapythom2014goto}{article}{
	title={About Got{\^o}'s method showing surjectivity of word maps},
	author={Elkasapy, Abdelrhman},
	author={Thom, Andreas},
	journal={Indiana University Mathematics Journal},
	volume={63},
	date={2014},
	number={5},
	pages={1553--1565}
}
\bib{fox1953free}{article}{
	title={Free differential calculus. I:\ Derivation in the free group ring},
	author={Fox, Ralph},
	journal={Annals of Mathematics},
	volume={57},
	number={3},
	pages={547--560},
	year={1953}
}
\bib{frickeklein1965vorlesungen}{book}{
	title={Vorlesungen {\"u}ber die Theorie der automorphen Funktionen. Band 1: Die gruppentheoretischen Grundlagen. Band II: Die funktionentheoretischen Ausf{\"u}hrungen und die Anwendungen},
	author={Fricke, Robert},
	author={Klein, Felix},
	series={Bibliotheca Mathematica Teubneriana},
	volume={4},
	publisher={Johnson Reprint Corporation, New York; B.\ G.\ Teubner Verlagsgesellschaft, Stuttgart},
	date={1965},
	pages={Band I: xiv+634 pp.; Band II: xiv+668}
}
\bib{fultonharris1991representation}{book}{
   author={Fulton, William},
   author={Harris, Joe},
   title={Representation theory},
   series={Graduate Texts in Mathematics},
   volume={129},
   publisher={Springer Science \& Business Media},
   date={1991},
   pages={xvi+551},
}
\bib{garionshalev2009commutator}{article}{
	title={Commutator maps, measure preservation, and $T$-systems},
	author={Garion, Shelly},
	author={Shalev, Aner},
	journal={Transactions of the American Mathematical Society},
	volume={361},
	number={9},
	pages={4631--4651},
	year={2009}
}
\bib{goldman2009trace}{article}{
	author={Goldman, William},
	title={Trace coordinates on Fricke spaces of some simple hyperbolic
		surfaces},
	conference={
		title={Handbook of Teichm\"{u}ller theory. Vol. II},
	},
	book={
		series={IRMA Lectures in Mathematics and Theoretical Physics},
		volume={13},
		publisher={European Mathematical Society, Z\"{u}rich},
	},
	date={2009},
	pages={611--684}
}
\bib{gonshawliebeckobrien2017unipotent}{article}{
	title={Unipotent class representatives for finite classical groups},
	author={Gonshaw, Samuel},
	author={Liebeck, Martin},
	author={O'Brien, Eamonn},
	journal={Journal of Group Theory},
	volume={20},
	number={3},
	pages={505--525},
	year={2017},
	publisher={De Gruyter}
}
\bib{gordeevkunyavskiiplotkin2016word}{article}{
	title={Word maps and word maps with constants of simple algebraic groups},
	author={Gordeev, Nikolai},
	author={Kunyavski{\u\i}, Boris},
	author={Plotkin, Eugene},
	journal={Doklady Mathematics},
	volume={94},
	number={3},
	pages={632--634},
	year={2016},
	organization={Springer Science \& Business Media}
}
\bib{gordeevkunyavskiiplotkin2018word}{article}{
	title={Word maps on perfect algebraic groups},
	author={Gordeev, Nikolai},
	author={Kunyavski{\u\i}, Boris},
	author={Plotkin, Eugene},
	journal={International Journal of Algebra and Computation},
	volume={28},
	number={8},
	pages={1487--1515},
	year={2018},
	publisher={World Scientific}
}
\bib{gulweiss2017dimension}{article}{
	title={On the dimension of matrix embeddings of torsion-free nilpotent groups},
	author={Gul, Funda},
	author={Wei{\ss}, Armin},
	journal={Journal of Algebra},
	volume={477},
	pages={516--539},
	year={2017},
	publisher={Elsevier}
}
\bib{guralnickliebeckobrienshalevtiep2018surjective}{article}{
	title={Surjective word maps and {B}urnside's $p^a q^b$ theorem},
	author={Guralnick, Robert},
	author={Liebeck, Martin}, 
	author={O'Brien, Eamonn},
	author={Shalev, Aner},
	author={Tiep, Pham Huu},
	journal={Inventiones mathematicae},
	volume={213},
	number={2},
	pages={589--695},
	year={2018},
	publisher={Springer Science \& Business Media}
}
\bib{guralnicktiep2015effective}{article}{
	title={Effective results on the Waring problem for finite simple groups},
	author={Guralnick, Robert},
	author={Tiep, Pham Huu},
	journal={American Journal of Mathematics},
	volume={137},
	number={5},
	pages={1401--1430},
	year={2015},
	publisher={The Johns Hopkins University Press}
}
\bib{horowitz1972characters}{article}{
	author={Horowitz, Robert},
	title={Characters of free groups represented in the two-dimensional
		special linear group},
	journal={Communications on Pure and Applied Mathematics},
	volume={25},
	date={1972},
	pages={635--649},
}
\bib{huilarsenshalev2015waring}{article}{
	title={The Waring problem for Lie groups and Chevalley groups},
	author={Hui, Chun Yin},
	author={Larsen, Michael},
	author={Shalev, Aner},
	journal={Israel Journal of Mathematics},
	volume={210},
	number={1},
	pages={81--100},
	year={2015},
	publisher={Springer Science \& Business Media}
}
\bib{huppert2013endliche}{book}{
	title={Endliche Gruppen I},
	author={Huppert, Bertram},
	volume={134},
	year={2013},
	publisher={Springer Science \& Business Media}
}
\bib{jamborliebeckobrien2013some}{article}{
	title={Some word maps that are non-surjective on infinitely many finite simple groups},
	author={Jambor, Sebastian},
	author={Liebeck, Martin},
	author={O'Brien, Eamonn},
	journal={Bulletin of the London Mathematical Society},
	volume={45},
	number={5},
	pages={907--910},
	year={2013},
	publisher={Wiley Online Library}
}
\bib{klyachkothom2017new}{article}{
	title={New topological methods to solve equations over groups},
	author={Klyachko, Anton},
	author={Thom, Andreas},
	journal={Algebraic \& Geometric Topology},
	volume={17},
	number={1},
	pages={331--353},
	year={2017},
	publisher={Mathematical Sciences Publishers}
}
\bib{larsen1997often}{article}{
	title={How often is a permutation an $n$'th power?},
	author={Larsen, Michael},
	journal={arXiv preprint math/9712223},
	year={1997}
}
\bib{larsen2004word}{article}{
	title={Word maps have large image},
	author={Larsen, Michael},
	journal={Israel Journal of Mathematics},
	volume={139},
	number={1},
	pages={149--156},
	year={2004},
	publisher={Springer Science \& Business Media}
}
\bib{larsenshalev2016distribution}{article}{
	title={On the distribution of values of certain word maps},
	author={Larsen, Michael},
	author={Shalev, Aner},
	journal={Transactions of the American Mathematical Society},
	volume={368},
	number={3},
	pages={1647--1661},
	year={2016}
}
\bib{larsenshalev2009word}{article}{
	title={Word maps and Waring type problems},
	author={Larsen, Michael},
	author={Shalev, Aner},
	journal={Journal of the American Mathematical Society},
	volume={22},
	number={2},
	pages={437--466},
	year={2009}
}
\bib{larsenshalev2017words}{article}{
	title={Words, Hausdorff dimension and randomly free groups},
	author={Larsen, Michael},
	author={Shalev, Aner},
	journal={Mathematische Annalen},
	pages={1--19},
	year={2017},
	publisher={Springer Science \& Business Media}
}
\bib{larsenshalevtiep2012waring}{article}{
	title={Waring problem for finite quasisimple groups},
	author={Larsen, Michael},
	author={Shalev, Aner},
	author={Tiep, Pham Huu},
	journal={International Mathematics Research Notices},
	volume={2013},
	number={10},
	pages={2323--2348},
	year={2012},
	publisher={Oxford University Press}
}
\bib{liebeckshalev2001diameters}{article}{
	title={Diameters of finite simple groups: sharp bounds and applications},
	author={Liebeck, Martin},
	author={Shalev, Aner},
	journal={Annals of Mathematics},
	pages={383--406},
	year={2001},
	publisher={JSTOR}
}
\bib{liebeck2015width}{article}{
	title={Width questions for finite simple groups},
	author={Liebeck, Martin},
	journal={Groups St Andrews 2013},
	volume={422},
	pages={51},
	year={2015},
	publisher={Cambridge University Press}
}
\bib{linnik1944least}{article}{
	title={On the least prime in an arithmetic progression. I.~The basic theorem},
	author={Linnik, Yuri Vladimirovich},
	journal={Recueil Math{\'e}matique. Nouvelle S{\'e}rie},
	volume={15},
	number={2},
	pages={139--178},
	year={1944}
}
\bib{lubotzky2014images}{article}{
	title={Images of word maps in finite simple groups},
	author={Lubotzky, Alexander},
	journal={Glasgow Mathematical Journal},
	volume={56},
	number={2},
	pages={465--469},
	year={2014},
	publisher={Cambridge University Press}
}
\bib{lueck2002L2}{book}{
   author={L{\"u}ck, Wolfgang},
   title={$L^2$-invariants: theory and applications to geometry and
   $K$-theory},
   volume={44},
   publisher={Springer Science \& Business Media},
   date={2002},
   pages={xvi+595}
}
\bib{magnus1980rings}{article}{
	author={Magnus, Wilhelm},
	title={Rings of Fricke characters and automorphism groups of free groups},
	journal={Rings of Fricke characters and automorphism groups of free groups},
	volume={170},
	date={1980},
	number={1},
	pages={91--103}
}
\bib{nickel2006matrix}{article}{
	title={Matrix representations for torsion-free nilpotent groups by Deep Thought},
	author={Nickel, Werner},
	journal={Journal of Algebra},
	volume={300},
	number={1},
	pages={376--383},
	year={2006},
	publisher={Elsevier}
}
\bib{nikolovschneiderthom2017some}{article}{
	title={Some remarks on finitarily approximable groups},
	author={Nikolov, Nikolay},
	author={Schneider, Jakob},
	author={Thom, Andreas},
	journal={Journal de l'{\'E}cole polytechnique -- Math{\'e}matiques},
	volume={5},
	year={2017},
	publisher={Journal de l'{\'E}cole Polytechnique}
}
\bib{pestov2008hyperlinear}{article}{
	title={Hyperlinear and sofic groups: a brief guide},
	author={Pestov, Vladimir},
	journal={Bulletin of Symbolic Logic},
	volume={14},
	number={04},
	pages={449--480},
	year={2008},
	publisher={Cambridge University Press}
}
\bib{schneider2019phd}{thesis}{
	author={Schneider, Jakob},
	title={On ultraproducts of compact quasisimple groups},
	type={PhD thesis}
	school={TU Dresden},
	year={2019},
	status={to appear on \url{http://www.qucosa.de}}
}
\bib{segal2005polycyclic}{book}{
	title={Polycyclic groups},
	author={Segal, Daniel},
	number={82},
	year={2005},
	publisher={Cambridge University Press}
}
\bib{stolzthom2014lattice}{article}{
	title={On the lattice of normal subgroups in ultraproducts of compact simple groups},
	author={Stolz, Abel},
	author={Thom, Andreas},
	journal={Proceedings of the London Mathematical Society},
	volume={108},
	number={1},
	pages={73--102},
	year={2014},
	publisher={Oxford University Press}
}
\bib{thangaduraivatwani2011least}{article}{
	title={The least prime congruent to one modulo $n$},
	author={Thangadurai, Ravindranathan},
	author={Vatwani, Akshaa},
	journal={American Mathematical Monthly},
	volume={118},
	number={8},
	pages={737--742},
	year={2011},
	publisher={Mathematical Association of America}
}
\bib{thom2013convergent}{article}{
   author={Thom, Andreas},
   title={Convergent sequences in discrete groups},
   journal={Canadian Mathematical Bulletin},
   volume={56},
   date={2013},
   number={2},
   pages={424--433}
}  
\bib{thomwilson2018some}{article}{
	title={Some geometric properties of metric ultraproducts of finite simple groups},
	author={Thom, Andreas},
	author={Wilson, John},
	journal={Israel Journal of Mathematics},
	volume={227},
	number={1},
	pages={113--129},
	year={2018},
	publisher={Springer Science \& Business Media}
}
\bib{traina1980trace}{article}{
	title={Trace polynomial for two-generator subgroups of ${\rm SL}(2,\mathbb{C})$},
	author={Traina, Charles R.},
	journal={Proceedings of the American Mathematical Society},
	volume={79},
	number={3},
	pages={369--372},
	year={1980}
}
\bib{vogt1889sur}{article}{
	author={Vogt, H.},
	title={Sur les invariants fondamentaux des \'{e}quations diff\'{e}rentielles
		lin\'{e}aires du second ordre},
	journal={Annales Scientifiques de l'\'{E}cole Normale Sup\'{e}rieure. Troisi\`eme S\'{e}rie},
	volume={6},
	date={1889},
	pages={3--71}
}
\bib{wall1963conjugacy}{article}{
	title={On the conjugacy classes in the unitary, symplectic and orthogonal groups},
	author={Wall, G.\ E.\ },
	journal={Journal of the Australian Mathematical Society},
	volume={3},
	number={1},
	pages={1--62},
	year={1963},
	publisher={Cambridge University Press}
}
\bib{wilson2009finite}{book}{
	title={The finite simple groups},
	author={Wilson, Robert},
	volume={251},
	year={2009},
	publisher={Springer Science \& Business Media}
}
\end{biblist}
\end{bibdiv}
\end{document}